\documentclass[11pt]{amsart}

\usepackage[draft]{changes}

\definechangesauthor[color=red]{zl}
\definechangesauthor[color=blue]{hh}
\definechangesauthor[color=orange]{hx}
\newcommand{\stkout}[1]{\ifmmode\text{\sout{\ensuremath{#1}}}\else\sout{#1}\fi}
\setdeletedmarkup{\stkout{#1}}

\usepackage{amssymb} 

\usepackage{graphics}
\usepackage{graphicx}

\usepackage{amsmath} 
\usepackage{amsthm}
\usepackage{amsfonts}
\usepackage{xcolor}
\usepackage[english]{babel}
\usepackage[margin=1.0in]{geometry}
\usepackage[colorlinks=true,
        linkcolor=blue]{hyperref}
\usepackage{bbm}
\usepackage{verbatim}
\usepackage[T1]{fontenc}

\numberwithin{equation}{section}

\newtheorem{prop}{Proposition}
\newtheorem{lemma}[prop]{Lemma}

\newtheorem{thm}[prop]{Theorem}
\newtheorem{cor}[prop]{Corollary}

\numberwithin{prop}{section}

\theoremstyle{definition}

\newtheorem{rmk}[prop]{Remark}

\definecolor{c1}{rgb}{0.2,0.4,0.5}
\definecolor{c2}{rgb}{0.1,0.3,0.5}
\definecolor{c3}{rgb}{0.2,0.7,0.5}
\usepackage{tikz}

\def \k {K\"ahler }
\newcommand{\oo}[1]{\overline{#1}}

\newcommand{\del}{\partial}
\newcommand{\bdel}{\bar{\partial}}

\newcommand{\gw}{\omega}

\newcommand{\ten}{\otimes}

\newcommand{\dbar}{\oo\partial}

\DeclareMathOperator{\Real}{Re}

\begin{document}

\title[Off-diagonal asymptotic  of Bergman kernels]{Off-diagonal asymptotic properties of Bergman kernels associated to analytic K\"ahler potentials}

\begin{abstract} We prove a new off-diagonal asymptotic of the Bergman kernels associated to tensor powers of a positive line bundle on a compact K\"ahler manifold. We show that if the K\"ahler potential is real analytic, then the Bergman kernel accepts a complete asymptotic expansion in a neighborhood of the diagonal of shrinking size $k^{-\frac14}$. These improve the earlier results in the subject for smooth potentials, where an expansion exists in a $k^{-\frac12}$ neighborhood of the diagonal. We obtain our results by finding upper bounds of the form $C^m m!^{2}$ for the Bergman coefficients $b_m(x, \bar y)$, which is an interesting problem on its own. We find such upper bounds using the method of \cite{BBS}. We also show that sharpening these upper bounds would improve the rate of shrinking neighborhoods of the diagonal $x=y$ in our results. In the special case of metrics with local constant holomorphic sectional curvatures, we obtain off-diagonal asymptotic in a fixed (as $k \to \infty$) neighborhood of the diagonal, which recovers a result of Berman \cite{Ber} (see Remark 3.5 of \cite{Ber} for higher dimensions). In this case, we also find an explicit formula for the Bergman kernel mod $O(e^{-k \delta} )$.

\end{abstract}


\author [Hezari]{Hamid Hezari}
\address{Department of Mathematics, UC Irvine, Irvine, CA 92617, USA} \email{\href{mailto:hezari@uci.edu}{hezari@uci.edu}}

\author[Lu]{Zhiqin Lu}
\address{Department of Mathematics, UC Irvine, Irvine, CA 92617, USA}
\email{\href{mailto:zlu@uci.edu}{zlu@uci.edu}}

\author [Xu]{Hang Xu}
\address{Department of Mathematics, Johns Hopkins University, Baltimore, MA 21218, USA}
\email{\href{mailto:hxu@math.jhu.edu}{hxu@math.jhu.edu}}

\maketitle

\section{Introduction}

Let $(L,h) \to M$ be a positive Hermitian holomorphic line bundle over a compact complex manifold of dimension $n$. The metric $h$ induces the \k form $\gw= -\tfrac{\sqrt{-1}}{2} \del \bdel \log(h)$ on $M$.  For $k$ in $\mathbb N$, let $H^0(M,L^k)$ denote the space of holomorphic sections of $L^k$. The {Bergman projection} is the orthogonal projection $\Pi_k: {L}^{2}(M,L^k) \to H^0(M,L^k)$ with respect to the natural inner product induced by the metric $h^k$ and the volume form $\frac{ \gw^n }{n!}$. The \emph{Bergman kernel} $K_k$, a section of $L^k\ten \bar{L^k}$, is the distribution kernel of $\Pi_k$.
Given $p \in M$,  let $(V, e_L)$ be a local trivialization of $L$ near $p$.   We write $| e_L |^2_{h}=e^{-\phi}$ and call $\phi$ a local \k potential. In the frame $ e_L^{k} \ten {\bar{e}_L^{k}}$, the Bergman kernel $K_k(x,y)$ is understood as a function on $V \times V$. We note that on the diagonal $x=y$, the function $K_k(x,x)e^{-k\phi(x)}$ is independent of the choice of the local frame, hence it is a globally defined function on $M$  called the \emph{Bergman function}, which is also equal to $| K_k(x, x)  |_{h^k}$. 

Zelditch \cite{Ze1} and Catlin \cite{Ca} proved that on the diagonal $x=y$, the Bergman kernel accepts a complete asymptotic expansion of the form
\begin{equation}\label{ZC}
K_k(x,x)e^{-k\phi(x)}\sim\frac{k^n}{\pi^n}\left(b_0(x,\bar{x})+\frac{b_1(x,\bar{x})}{k}+\frac{b_2(x,\bar{x})}{k^2}+\cdots\right).
\end{equation}
\textit{Near the diagonal}, i.e. in a $\frac{1}{\sqrt{k}}$-neighborhood of the diagonal, one has a scaling asymptotic expansion for the Bergman kernel (see \cite{ShZe, MaMaBook, MaMaOff, LuSh, HKSX}).  For $d(x, y) \gg \sqrt{\frac{\log k}{{k}}}$, where $d$ is the Riemannian distance induced by $\omega$, no useful asymptotics are known. However,  there are off-diagonal upper bounds of Agmon type 
\begin{equation}\label{Agmon} \left | K_k(x, y) \right |_{h^k}  \leq C k^n e^{- c \sqrt{k} d(x, y)}, \end{equation}
proved for smooth metrics by Ma-Marinescu \cite{MaMaAgmon}\footnote{Estimate \eqref{Agmon} was first stated in \cite{Bern}, which is analogous to earlier results \cite{Ch1, Del, Lindholm} in various settings.} In fact as shown by Christ in \cite{Ch3}, away from the diagonal one has better decay estimates. More precisely, for any $\delta>0$ there exist $K$ and a function $f(k) \to \infty$ as $k \to \infty$ such that\footnote{In \cite{LS}, similar decays are obtained  in the non-compact setting under the assumption of Ricci curvature being uniformly bounded from below. See also page 64 of \cite{Seto} for Agmon estimates in this setting.}
$$ \qquad  \left | K_k(x, y) \right |_{h^k} \leq C e^{- f(k) \sqrt{k \log k}}, \quad \text{for} \; d(x, y) > \delta \; \text{and} \; k >K. $$
When the metric $h$ is real analytic the much improved estimate 
$$ \left | K_k(x, y) \right |_{h^k} \leq C k^n e^{- c k d^2(x, y)},$$ holds\footnote{We were unable to find a proof of this estimate in the literature.} (see estimate 6.2 in Remark 6.6 of \cite{Ch}).  The goal of this article is to prove an asymptotic expansion in a $k^{-1/4}$ neighborhood of the diagonal in the real analytic case. In particular, we show that in the real analytic case, uniformly for all sequences $x_k$ and $y_k$ with $d(x_k, y_k) \leq k^{-1/4}$, we have
$$  \left | K_k(x_k, y_k) \right |_{h^k} \sim \frac{k^n}{\pi^n} e^{- \frac{k D(x_k, y_k)}{2}}, \quad \text{as} \quad k \to \infty,$$ where $D(x, y)$ is Calabi's diastasis function \eqref{Diastatis}, which is controlled from above and below by $d^2(x, y)$.  
Before we state the results we must also mention that in the literature\footnote{See for example page 214 of \cite{BBS}, page 66 of \cite{KS}, and page 28 of \cite{DLM}.  } there is an \textit{ill} off-diagonal asymptotic expansion for the Bergman kernel of the form 
\begin{equation}\label{ill}
K_k(x,y)=e^{k\psi(x,\bar y)}\frac{k^n}{\pi^n}\left (1+\sum _{j=1}^{ {N-1}}\frac{b_j(x, \bar y)}{k^j}\right )+e^{k\left(\frac{\phi(x)}{2}+\frac{\phi(y)}{2}\right)}k^{-N+ n}O_N(1),
\end{equation} 
which holds for all $d(x, y) \leq \delta$ for some $\delta >0$. Here, $\psi(x, \bar y)$ and $b_j(x, \bar y)$ are holomorphic extensions of $\phi(x)$ and $b_j(x, \bar x)$ from \eqref{ZC}. However, note that this expansion is only useful when the term $e^{k\left(\frac{\phi(x)}{2}+\frac{\phi(y)}{2}\right)}k^{-N+ n}$ is a true remainder term, i.e. it is less than the principal term $k^n e^{k \psi(x,\bar y)}$ in size, which holds only when $d(x, y) \leq {C_N} \sqrt{\frac{\log k}{{k}}}$. Thus, to obtain an asymptotic when  $d(x, y) \gg \sqrt{\frac{\log k}{{k}}}$, we need a quantitative description of  the term $O_N(1)$  and the coefficients $b_j$ in \eqref{ill} which would allow $N$ to depend on $k$ accordingly. 

We now state our main result and its corollaries. \begin{thm}\label{Main}
 Assume that the local \k potential $\phi$ is real analytic in $V$. Let $\psi(x, z)$ be the holomorphic extension of $\phi(x)$ near the diagonal obtained by \textit{polarization}, i.e., $\psi(x,z)$ is holomorphic and $\psi(x,\bar{x})=\phi(x)$. Also let $b_m(x , z)$ be the holomorphic extensions of the Bergman kernel coefficients $b_m(x, \bar x)$ in the on-diagonal expansion \eqref{ZC}.
Then there exist positive constants $\delta$ and $C$, and an open set $U \subset V$ containing $p$, such that for $N_0(k)=[(\frac{k}{C})^{\frac{1}{2}}]$ and uniformly for any $x,y\in U$, we have
\begin{equation*}
K_k(x,y)=e^{k\psi(x,\bar y)}\frac{k^n}{\pi^n}\left (1+\sum _{j=1}^{ N_0(k)-1}\frac{b_j(x, \bar y)}{k^j}\right )+e^{k\left(\frac{\phi(x)}{2}+\frac{\phi(y)}{2}\right)}e^{-\delta k^{\frac{1}{2}}}O(1). 
\end{equation*} 
Moreover, one can differentiate this expansion. More precisely,
\begin{equation*}
D^\alpha K_k(x,y)=D^\alpha\left(e^{k\psi(x,\bar y)}\frac{k^n}{\pi^n}\left (1+\sum _{j=1}^{ N_0(k)-1}\frac{b_j(x, \bar y)}{k^j}\right )\right)+k^{|\alpha|}\alpha! e^{k\left(\frac{\phi(x)}{2}+\frac{\phi(y)}{2}\right)}e^{-\delta k^{\frac{1}{2}}}O(1),
\end{equation*} 
for any differential operator $D^\alpha$  of order $\alpha$ with respect to $x$ and $y$.
\end{thm}

As a first corollary of this theorem, we get a complete asymptotic expansion in a $k^{-\frac14}$ neighborhood of the diagonal.

\begin{cor}\label{complete asymptotics}
Given the same assumptions and notations as in the above theorem, there exist positive  constants $C$ and $\delta$, and an open set $U \subset V$ containing $p$, such that for all $k$ and $N\in \mathbb{N}$, we have for all $x,y\in U$ satisfying $d(x,y)\leq \delta k^{-\frac{1}{4}}$,
\begin{equation*}
K_k(x,y)=e^{-k\psi(x,\bar{y})}\frac{k^n}{\pi^n}\left(1+\sum_{j=1}^{{N-1}}\frac{b_j(x,\bar{y})}{k^j}+\mathcal{R}_N(x, \bar y,k)\right),
\end{equation*}
where 
\begin{equation*}
	\left|\mathcal{R}_N(x, \bar y ,k)\right|\leq {\frac{C^{N}N!^2}{k^N}}.
\end{equation*}
\end{cor}

As another corollary, we obtain the following off-diagonal asymptotic in terms of Calabi's diastasis \cite{Cal} function defined by
\begin{equation}\label{Diastatis} 
D(x, y)= \phi(x) + \phi(y) - \psi(x, \bar y) - \psi(y, \bar x).  
\end{equation} We point out that near a given point $p \in M$, we have  $D(x, y) = |x -y|_p^2 + O(|x-p|_p^3+|y-p|_p^3)$, where $|z|^2_p:= \sum_{ i, j=1}^n \phi_{i \bar j}(p)z_i\oo{z_j }$. If we use \textit{Bochner coordinates} at $p$ (introduced in \cite{Boc}), in which the \k potential admits the form $\phi(x)=|x|^2+O(|x|^4)$, we have $D(x, y) = |x -y|_p^2 + O(|x-p|_p^4+|y-p|_p^4)$. 
\begin{cor}\label{Cor1}
Under the same assumptions and notations (and the same $\delta$ and same $U$) as in Theorem \ref{Main},  we have uniformly for all $x , y \in U$ satisfying $ D(x, y) \leq  \frac12 \delta k^{-\frac12}$,
\begin{equation}\label{Log}
\frac{1}{k}\log \left | K_k (  x, y ) \right | _{h^k}=-\frac{D(x, y)}{2} + \frac{n \log k}{k} -\frac{n\log\pi}{k}+ O \left (\frac{1}{k^2} \right ).
\end{equation}
\end{cor}
The following scaling asymptotic is then immediate:
\begin{cor}\label{Cor2} In Bochner coordinates at $p$, we have uniformly for all $u,v \in \mathbb{C}^n$ with $| u|_p$ and $|v|_{p} < \frac{\sqrt{\delta}}{3}$,
\begin{equation*}
\frac{1}{k^{\frac12}}\log \left | K_k \left (\frac{u}{k^{\frac14}},\frac{v}{k^{\frac14} }\right ) \right |_{h^k}=-\frac{|u-v|_p^2}{2} + \frac{n \log k}{k^{\frac12}}-\frac{n\log\pi}{k^{\frac{1}{2}}} + O \left ( \frac{1}{k^\frac32}\right ).
\end{equation*}
\end{cor}

One of the key ingredients in our proofs is the following estimate on the Bergman kernel coefficients $b_m(x, z)$. We emphasize again that $b_m(x , z)$ are the holomorphic extensions of the Bergman kernel coefficients $b_m(x, \bar x)$ appearing in the on-diagonal expansion \eqref{ZC} of Zelditch \cite{Ze1} and Catlin \cite{Ca}.
\begin{thm}\label{MainLemma}
Assume the \k potential $\phi$ is real analytic in some neighborhood $V$ of $p$. Then, there exists a neighborhood $U\subset V$ of $p$, such that for any $m \in \mathbb N$ we have
\begin{equation*}
\|b_m(x,z)\|_{L^\infty(U\times U)} \leq C^{m}m!^{2},
\end{equation*}
where $C$ is a constant independent of $m$. 
\end{thm}

\begin{rmk} We conjecture that in the real analytic case
\begin{equation}\label{conj}
\|b_m(x,z)\|_{L^\infty(U\times U)} \leq C^{m}m! \, .
\end{equation}
As we show in this paper, if this conjecture holds true, then all of the above results can be improved accordingly. In particular, the quantities $N_0(k)= (k /C)^\frac12$ and $e^{- \delta k^\frac12}$ in the remainder estimate of Theorem \ref{Main} would be replaced by $k/C$ and $e^{- \delta k}$, moreover Corollary \ref{Log} would hold for all $D(x, y) \leq \frac{\delta}{2}$.  However we are unable to prove this conjecture for general real analytic \k metrics using our method, which is based  on a recursive formula of \cite{BBS}. In Section \ref{optimal}, we discuss the optimality and limitations of this method. One can also check, with a bit of more effort, that the recursive formula of \cite{Loi} (hence equivalently the method of \cite{Xu}) would also provide the same estimates as our Theorem \ref{MainLemma}. In addition, in Theorem \ref{Transport}, we find a parametrix representation of the Bergman kernel by means of transport equations similar to those for the wave equation, but as we discuss in Remark \ref{FinalRemark}, this method would also give the bounds $C^m (m!)^2$. Another way to explore this conjecture is to use the method of peak sections. It would be interesting to confirm the bounds $C^m (m!)^2$ using this method. We must also mention that in \cite{LiuLu2}, it is claimed that in the analytic case \begin{equation*}
\|b_m(x,\bar x)\|_{L^\infty(U)} \leq C^{m},
\end{equation*}
however the proof contains some errors. In fact we do not expect the upper bounds $C^m$  to be correct in general although we do not have any counterexamples. We doubt the bounds $C^m$ because  by \cite{LuTian} the leading term in $b_m(x,\bar{x})$ is $\frac{m}{(m+1)!}\Delta^{m-1}\rho(x)$ where $\rho$ is the scalar curvature, so when the metric is real analytic we have $\frac{m}{(m+1)!}\Delta^{m-1}\rho(x)\approx C^m m!$. Nevertheless, as we shall see below, the bounds $C^m$ hold trivially in the case of constant holomorphic sectional curvatures. 

\end{rmk}

\subsection{Metrics with local constant holomorphic sectional curvatures} As a special case, if we assume that the \k metric $\omega=\frac{\sqrt{-1}}{2}\partial\dbar \phi$ has constant holomorphic sectional curvatures in $V$, then $\phi(x)$ is analytic in $V$ and there exists $U \subset V$ containing $p$ such that for $x, z \in U$, $b_m(x,z)$ are all constants and vanish for $m>n=\dim M$. We now state the improved results.

\begin{thm}\label{Constant1}
Assume the \k metric $\omega$ has constant holomorphic sectional curvatures $c$ in $V$. Then there exist a positive constant $\delta$, an open set $U \subset V$ containing $p$, and constants $b_1, \dots, b_n$ only dependent on $n$ and $c$, such that uniformly for all $x, y \in U$, we have
\begin{equation*}
K_k(x,y)=e^{k\psi(x,\bar y)}\frac{k^n}{\pi^n}\left (1+\sum _{j=1}^{ n}\frac{b_j}{k^j}\right )+e^{k\left(\frac{\phi(x)}{2}+\frac{\phi(y)}{2}\right)}e^{-\delta k}O(1).
\end{equation*}
Moreover, we have an explicit formula for the amplitude given by
\begin{align*} \label{gamma}
		1+\sum_{j=1}^n\frac{b_j}{k^j}=
		\begin{cases}
		\frac{c^n}{k^n}\frac{\Gamma (\frac{k}{c}+n+1)}{\Gamma(\frac{k}{c}+1)}=\frac{c^n}{k^n} \left (\frac{k}{c}+n \right ) \left (\frac{k}{c}+n-1 \right ) \dots \left (\frac{k}{c}+1 \right ), & \mbox{when } c\neq 0, \\
		1, & \mbox{when } c=0. 
		\end{cases}
	\end{align*}

In particular, if the \k metric $\omega$ has global constant holomorphic sectional curvature, we have
\begin{equation}\label{BFAsymptotic}
\left\|B_k(x)-\frac{k^n}{\pi^n}\left (1+\sum _{j=1}^{ n}\frac{b_j}{k^j}\right )\right\|_{L^\infty(M)}=e^{-\delta k}O(1),
\end{equation}
where $B_k(x):=\left|K_k(x,x)\right|_{h^k}=K_k(x,x)e^{-k\phi(x)}$ is the Bergman function.
\end{thm}

Estimate \eqref{BFAsymptotic} was previously obtained by Berman \cite{Ber} for Riemann surfaces (see also Remark 3.5 of \cite{Ber} for a sketch of a proof for all dimensions). 

Accordingly, we get the following corollary. \begin{cor}\label{Constant2}
Under the same assumptions and notations as in the above theorem,  we have uniformly for all $x , y \in U$ satisfying $ D(x, y) \leq  \frac{1}{2} \delta $,
\begin{equation*}
\frac{1}{k}\log \left | K_k (  x, y ) \right | _{h^k}=-\frac{D(x, y)}{2} + \frac{n \log k}{k}-\frac{n\log\pi}{k} + O \left (\frac{1}{k^2} \right ).
\end{equation*}
\end{cor}
 

There is a huge literature on Bergman kernels on compact complex manifolds. Before closing the introduction we only list some related work that were not cited above: \cite{BoSj, En, Ch, Loi, LuTian, MaMa, Liu, LiuLu1, LuZe, Ze3}. Applications of the Bergman kernel, and the closely related Szeg\"o kernel, can be found in \cite{Do}, \cite{BSZ}, \cite{ShZe}, \cite{YZ}. The book of Ma and Marinescu \cite{MaMaBook} contains an introduction to the asymptotic expansion of the Bergman kernel and its applications. See also the book review \cite{ZeBookReview} for more on the applications of Bergman kernels.  

\subsection{Organization of the paper} In Sections \ref{Sec Local} and \ref{Sec MethodBBS}, we follow the construction of local Bergman kernel in \cite{BBS}, but we obtain precise estimates for the error term by using the growth rate of Bergman coefficients $b_m(x,z)$ provided by Theorem \ref{MainLemma}. In Section \ref{Sec Global}, we give the proofs of Theorem \ref{Main} and Corollaries  \ref{complete asymptotics} and \ref{Cor1}. The proof of Theorem \ref{MainLemma} will be given in Section \ref{Sec ProofofMainLemma}. Section \ref{optimal} discusses the optimality of our bounds on Bergman coefficients. Section \ref{Sec ConstantCurvature} concerns the special case of constant holomorphic sectional curvatures. Finally, in Section \ref{Sec NewFormula} we derive transport equations for the amplitude of the parametrix of \cite{BBS}, which is analogous to transport equations for the parametrices of the wave kernel. 

\section{Local Bergman kernels}\label{Sec Local}

In \cite{BBS}, by using \emph{good} complex contour integrals, Berman-Berndtsson-Sj\"ostrand constructed  \textit{local reproducing kernels} (mod $e^{-k\delta}$) for $U=B^n(0,1)\subset \mathbb{C}^n$, which reproduce holomorphic sections in $U$ up to  $e^{-k\delta}$ error terms.  These kernels  are in general not holomorphic. By allowing more flexibility in choosing the amplitudes in the integral, the authors modified  these local reproducing kernels  to  local Bergman kernels, which  means that they are \emph{holomorphic}  local reproducing kernels mod $O(k^{-N})$. The global Bergman kernels are then approximated  using the standard  H\"ormander's $L^2$ estimates. 

Throughout this paper, we assume that $\phi$ is real analytic in a small open neighborhood $V\subset M$ of a given point $p$. Let $B^n(0,r)$ be the ball of radius $r$ in $\mathbb C^n$. We identify $p$ with $0\in \mathbb{C}^n$ and $V$ with the ball $B^n(0,3)\subset\mathbb{C}^n$ and denote $U=B^n(0, 1)$. Let $e_L$ be a local holomorphic frame of $L$ over $V$ as introduced in the introduction. 
For each positive integer $k$, we denote $H_{k\phi}(U)$ to be the inner product space of $L^2$-holomorphic functions on $U$ with respect to
\begin{equation*}
\left ( u, v \right)_{k\phi}=\int_U u \bar v \, e^{-k\phi}d \text{Vol},
\end{equation*} 
where $d\text{Vol}=\frac{\omega^n}{n!}$ is the natural volume form induced by the \k form $\omega= \frac{\sqrt{-1}}{2} \partial \bar \partial \phi$. 
So the norm of $u \in H_{k\phi}(U)$ is given by
\begin{equation*}
\|u\|^2_{k\phi}=\int_U |u|^2e^{-k\phi}d\text{Vol}. 
\end{equation*}
Let $\chi\in C_0^\infty(B^n(0,1))$ be a smooth cut-off function such that $\chi=1$ in $B^n(0,\frac{1}{2})$ and vanishes outside $B^n(0,\frac{3}{4})$.  The following result  gives a refinement of the the result of \cite{BBS} by giving a more precise estimate for the error term when the K\"ahler potential is real analytic. The main ingredient of the proof is Theorem \ref{MainLemma}, whose proof is delayed to Section \ref{Sec ProofofMainLemma}.
\begin{prop}\label{LocalBergmanKernel}
For each $N \in \mathbb N$, there exist $K_{k,x}^{(N)}(y) \in H_{k\phi}(U)$ and a positive constant $C$ independent of $N$ and $k$, such that for all $u \in H_{k\phi}(U)$ we have 
\begin{align}\label{eq 5.2}
\forall x \in B^n(0, 1/4): \quad u(x)=\left(\chi u, K^{(N)}_{k,x}\right)_{k\phi}+k^ne^{\frac{k\phi(x)}{2}}\mathcal R_{N+1}(\phi, k) \|u\|_{k\phi},
\end{align}
where
\begin{equation}\label{eq 5.3}
\quad |\mathcal R_{N+1}(\phi, k)| \leq \frac{C^{N+1}(N+1)!^2}{k^{N+1}}.
\end{equation}
The function $K_{k,x}^{(N)}$ is called a local Bergman kernel of order $N$.
\end{prop}
\begin{rmk}
In \cite{BBS}, only the qualitative estimate $\mathcal R_{N+1}(\phi, k)= O_N(\frac{1}{k^{N+1}})$ is given.
\end{rmk}
To prove this proposition we first need to recall the techniques of \cite{BBS}. 

\subsection{Review of the method of Berman-Berndtsson-Sj\"ostrand} \label{ReviewBBS} The main idea is to construct the local holomorphic reproducing kernel (also called local Bergman kernel) by means of the calculus of contour pseudo-differential operators (contour $\Psi$DO for short) introduced by Sj\"ostrand \cite{Sj}.  Before we introduce the notion of contour integrals we present some notations and definitions. 

Suppose $\phi(x)$ is analytic in $V=B^n(0, 3)$. Without the loss of generality we assume that the radius of convergence of it power series in terms of $x$ and $\bar x$ is $3$.  By replacing $\phi(x)$ by $\phi(x) - \phi(0)$, we can assume that $\phi(0)=0$.  We  then denote $\psi(x,z)$ to be the holomorphic extension of $\phi(x)$ by replacing $\bar x$ with $z$. This procedure is called polarization. One can easily verify that $ \psi(x, z)$ satisfies the formal definition of holomorphic extension, namely
\begin{itemize}
\item $\psi(x, z)$ is holomorphic in $B^n(0, 3) \times B^n(0, 3)$.
\item $\psi(x, \bar x) = \phi(x)$. 
\end{itemize}   
Moreover, since $\phi(x)$ is real-valued, we have $ \oo{\psi(x, z)} = \psi( \bar z, \bar x)$. We also define
\begin{equation}\label{theta}
	\theta(x,y,z)=\int_0^1(D_x\psi)(tx+(1-t)y,z)dt,
\end{equation}
where the differential operator $D_x$ is the gradient operator defined by
\begin{align*}
D_x&=(D_{x_1},D_{x_2},\cdots,D_{x_n}).
\end{align*}
Note that $\theta(x, x, z)= \psi_x(x, z)$. 
It is easy to prove that the Jacobian of the map $(x,y,z)\rightarrow (x,y,\theta)$ at $(x,y,z)=(0,0,0)$ is non-singular.  Thus the map is actually  a biholomorphic map between two neighborhoods of the origin of $\mathbb C^{3n}$. As a result, we can use $(x,y,z)$ or $(x,y,\theta)$ as local coordinates interchangeably. Without  loss of generality we can assume that $(x, y, z) \in B^n(0, 3) \times B^n(0, 3) \times B^n(0, 3)$ and  $\theta \in W$, where
$$W =\theta \left ( B^n(0, 3) \times B^n(0, 3) \times B^n(0, 3)\right ). $$ Note that $W$ contains the origin because by our assumption $\phi(0)=0$. 

A fundamental idea of \cite{BBS} is to use the estimate
\begin{equation}\label{BBS 2.1}
u(x)= c_n \left(\frac{k}{2\pi}\right)^n\int_\Lambda e^{k\,\theta\cdot (x-y)}u(y)\chi(y)d\theta\wedge dy
+O(e^{-k\delta})e^{\frac{k\phi(x)}{2}}\|u\|_{k\phi}
\end{equation}
which holds uniformly for $x \in B^n(0, \frac14)$, for any holomorphic function $u$ defined on $ B^n(0, 1)$. Here, 
$c_n=i^{-n^2}$, $\delta$ is a positive constant, and $\Lambda=\{(y,\theta):\theta=\theta(x,y)\}$ is a \textit{good contour}, which means that there exists $\delta>0$ such that for any $x,y$ in a neighborhood of the origin, 
\begin{align}\label{GoodContour}
	2\Real\theta\cdot (x-y)\leq -\delta|x-y|^2-\phi(y)+\phi(x).
\end{align} 
One can easily verify that 
\begin{equation} \label{contour} \Lambda=\{(y,\theta):\theta=\theta(x,y,\bar{y})\}, \end{equation} with $\theta(x,y,z)$ defined by \eqref{theta}, is a good contour by observing that
\begin{equation*}
\theta\cdot (x-y)=\psi(x,\bar{y})-\psi(y,\bar{y}).
\end{equation*} 
To put \eqref{BBS 2.1} into a useful perspective, one should think of the integral in \eqref{BBS 2.1} as a contour $\Psi$DO defined as follows. Let ${a=a(x,y,\theta,k)}$ be a holomorphic symbol in $B^n(0, 3) \times B^n(0, 3) \times W$, with an asymptotic expansion of the form
\begin{equation*}
a(x,y,\theta,k)\sim a_0(x,y,\theta)+\frac{a_1(x,y,\theta)}{k}+\frac{a_2(x,y,\theta)}{k^2}+\cdots \; .
\end{equation*}
For simplicity, we will suppress the dependency on $k$ and write $a=a(x,y,\theta)$.

A $\Psi$DO associated to a good contour $\Lambda$ and an amplitude $a(x, y, \theta)$, is an operator on $C^\infty_0(U)$ defined by 
$$\text{Op}_\Lambda(a)\, u = c_n \left(\frac{k}{2\pi}\right)^n \int_\Lambda e^{k\,\theta\cdot (x-y)}a(x, y, \theta)\,  u(y)\, d\theta\wedge dy.$$ Thus in this language \eqref{BBS 2.1} means that for $x \in B^n(0, 1/4)$
$$(\chi u)(x)= \text{Op}_\Lambda(1) (\chi u) 
+O(e^{-k\delta})e^{\frac{k\phi(x)}{2}}\|u\|_{k\phi}.$$
Roughly speaking this says that $\text{Op}_\Lambda (1)$ is the identity operator mod $O(e^{-k\delta})$.  
We define the integral kernel $K_{k, x}(y)$ of $\text{Op}_\Lambda(a)$ with respect to the inner product $( \cdot, \cdot )_{k \phi}$, by 
$$\text{Op}_\Lambda(a) u =\left(u, K_{k,x}\right)_{k\phi}. $$ The first observation is that the kernel $K_{k, x}(y)$ of $\text{Op}_\Lambda(1)$, associated to the contour \eqref{contour}, is not holomorphic. The idea of \cite{BBS} is to replace $\text{Op}_\Lambda(1)$ by $\text{Op}_\Lambda(1+a)$ where $a(x, y, \theta)$ is a \textit{negligible amplitude} and the kernel of $\text{Op}_\Lambda(1+a)$ is holomorphic. An amplitude $a(x, y, \theta)$ is negligible if 
$$\text{Op}_\Lambda(a) (\chi u)= O(k^{-\infty})e^{\frac{k\phi(x)}{2}}\|u\|_{k\phi}. $$ 
To find a suitable condition for negligible amplitudes one formally writes 
$$ \text{Op}_\Lambda(a) = \text{Op}_\Lambda(S a |_{x=y}),$$
where $S$ is a standard operator that is used in microlocal analysis to turn a symbol $a(x, y, \theta)$ of a $\Psi$DO to a symbol of the form $\tilde a (x, \theta)$. The operator $S$ is formally defined by
\begin{equation*}
S=e^{\frac{D_\theta \cdot D_y}{k}}=\sum_{m=0}^\infty\frac{(D_\theta\cdot D_y)^m}{m!k^m}.
\end{equation*}
Then an amplitude $a$ is negligible if $S a |_{x=y} \sim 0$ as a formal power series. This implies that there exists a holomorphic vector field $A(x, y, \theta)$ with formal power series \begin{equation*}
A(x,y,\theta)\sim
A_0(x,y,\theta)+\frac{A_1(x,y,\theta)}{k}+\frac{A_2(x,y,\theta)}{k^2}+\cdots.
\end{equation*} such that 
\begin{equation}\label{eq 5.11}
Sa\sim k(x-y)\cdot SA.
\end{equation}
By a straightforward calculation, it can be seen that this is equivalent to
\begin{equation}\label{nabla}
    a= \nabla A:=D_\theta \cdot A+k(x-y)\cdot A. 
\end{equation}
By comparing coefficients, \eqref{nabla} is equivalent to the following relations between $a_m$ and $A_m$:
\begin{equation}\label{Dtheta}
a_m=D_\theta\cdot A_m+(x-y)\cdot A_{m+1}\; .
\end{equation}
Here $a_m(x,y,\theta)$ are holomorphic functions and $A_m(x,y,\theta)$ are holomorphic vector fields in $\mathbb C^n$, defined on $B^{n}(0,3) \times B^{n}(0,3) \times W$. 

Next, we observe that the integral kernel of $\text{Op}_\Lambda(1+a)$ is holomorphic if
\begin{equation}\label{aB}
1+a(x,y,\theta)\sim B(x,z(x,y,\theta))\Delta_0(x,y,\theta),
\end{equation}
where
\begin{equation*}
	\Delta_0(x,y,\theta)=\frac{\det \psi_{yz}(y,z)}{\det\theta_z(x,y,z)}=\det\partial_\theta\psi_y(x,y,\theta),
\end{equation*}
and $B(x, z)$ is holomorphic and has an asymptotic expansion of the form
\begin{equation}\label{B} B(x, z) \sim b_0(x,z)+\frac{b_1(x,z)}{k}+\frac{b_2(x,z)}{k^2}+\cdots, \end{equation} where  $b_m(x, z)$ are holomorphic.  In fact, as it turns out, $b_m(x, z)$ are the holomorphic extensions of $b_m(x, \bar x)$, the Bergman kernel coefficients of the on-diagonal asymptotic expansion of Zelditch-Catlin \eqref{ZC}. 

If the amplitude $a$ is negligible, then by applying $S( \cdot) |_{x=y}$ to both sides of \eqref{aB}, we get 
$$S \left( B(x,z(x,y,\theta))\Delta_0(x,y,\theta) \right )|_{x=y} \sim 1.$$ 
From this, one gets the following recursive equations for Bergman kernel coefficients $b_m(x, z)$, which will play a key role in the proof of Theorem \ref{MainLemma}:
\begin{equation}\label{Recursive1}
	b_m(x,z(x,x,\theta))=-\sum_{l=1}^m\frac{(D_y\cdot D_\theta)^l}{l!}\big(b_{m-l}\left(x,z(x,y,\theta)\right)\Delta_0(x,y,\theta)\big)|_{y=x}.
\end{equation}
Additionally, by comparing the coefficients on both sides of \eqref{aB}, we have the following relations between $a_m$ and $b_{m}$:
\begin{equation}\label{eq 5.7}
a_m(x,y,\theta)=
\begin{cases}
\Delta_0(x,y,\theta)-1 & \mbox{ when }m=0,\\
b_m(x,z(x,y,\theta))\Delta_0(x,y,\theta) & \mbox{ when }m\geq 1.
\end{cases}
\end{equation}
These equations will be useful in estimating $a_m$ in terms of the bounds on $b_m$ from Theorem \ref{MainLemma}. 

We are now prepared to prove Proposition \ref{LocalBergmanKernel}.

\section{The remainder estimates and the proof of Proposition~\ref{LocalBergmanKernel}}\label{Sec MethodBBS}

Let $a_m$, $A_m$, and $b_m$ be given by \eqref{eq 5.7}, \eqref{Dtheta}, and \eqref{B}. Define $a^{(N)}$, $A^{(N)}$, and $B^{(N)}$ to be the partial sums of $a$, $A$, and $B$ up to order ${k^{-N}}$. We claim that
\begin{equation}\label{LocalBergmanKernel2}
u(x)=\text{Op}_\Lambda \left (1+a^{(N)} \right ) ( \chi u)
+k^n\mathcal R_{N+1}( \phi, k) e^{\frac{k\phi(x)}{2}}\|u\|_{k\phi},
\end{equation} where uniformly for $x \in B(0, \frac14 )$ we have
\begin{equation}\label{R}
|\mathcal R_{N+1}( \phi, k)| \leq\frac{C^{N+1}(N+1)!^2}{k^{N+1}},
\end{equation}
and the integral kernel of $\text{Op}_\Lambda \left (1+a^{(N)} \right )$  is holomorphic. The complex conjugate of this kernel is given by 
\begin{equation}\label{KandB}
\oo{K^{(N)}_{k,x}(y)}=\left(\frac{k}{\pi}\right)^n  e^{k\psi(x,\bar{y})}\left(1+ a^{(N)}(x, y, \theta(x, y, \bar y)\right) \Delta_0(x, y, \theta(x, y, \bar y))^{-1},
\end{equation}
which by the relation \eqref{aB} is reduced to 
$$\oo{K^{(N)}_{k,x}(y)}=\left(\frac{k}{\pi}\right)^n  e^{k\psi(x,\bar{y})} B^{(N)}(x, \bar y). $$
Hence  ${K^{(N)}_{k,x}(y)}$ is holomorphic in $y$ because $B(x, z)$ is holomorphic. 

In the light of \eqref{BBS 2.1}, to prove \eqref{LocalBergmanKernel2} it suffices to show that
$$\forall x \in B(0, \frac14): \quad \left | \text{Op}_\Lambda \left (a^{(N)} \right )( \chi u)(x)  \right | \leq \frac{C^{N+1}(N+1)!^2}{k^{N+1-n}} e^{\frac{k\phi(x)}{2}}\|u\|_{k\phi}. $$
By definition,
$$\text{Op}_\Lambda \left (a^{(N)} \right )( \chi u)(x)= c_n \left(\frac{k}{2\pi}\right)^n\int_\Lambda e^{k\theta\cdot (x-y)}u(y)\chi(y)\, a^{(N)} \, d\theta\wedge dy . $$
We observe that by the definition of $\nabla$ in \eqref{nabla},
\begin{equation*}
a^{(N)}-\nabla \left(A^{(N+1)}\right)
=-\frac{D_\theta\cdot A_{N+1}}{k^{N+1}}.
\end{equation*} 
Then it is easy to see that using integration by parts (see for example the proof of Proposition 2.2 in \cite{BBS}), we get
\begin{align*}\label{a^(N)}
\int_\Lambda e^{k\theta\cdot (x-y)}u(y)\chi(y)  a^{(N)}d\theta\wedge dy=
& -\int_\Lambda d\chi\wedge u(y)e^{k\theta\cdot(x-y)}A^{(N+1)}\wedge dy \\
& -\int_\Lambda e^{k\theta\cdot (x-y)}u(y)\chi(y)\frac{D_\theta\cdot A_{N+1}}{k^{N+1}}d\theta\wedge dy.
\end{align*} In the first integral, we have  identified  the $n$-vector $A$ as an $(n-1, 0)$ form defined by $ A= \sum_{j=1}^n A_j \widehat{d\theta_j}$, where $\widehat{d\theta_j}$ is the wedge product of all $\{d \theta_k\}_{k \neq j}$ such that $d\theta_j \wedge \widehat{d\theta_j} = d\theta$. 

We now estimate the two integrals on the right hand side of the above equality. For the first integral, since $d\chi(y)=0$ for $y\in B^{n}(0,\frac{1}{2})$,  $|x-y|\geq \frac{1}{4}$ for $x\in B^n(0,\frac{1}{4})$. Since $\theta$ satisfies \eqref{GoodContour}, by changing $\delta$ to a smaller constant, the integrand of the first integral is bounded by some constant times
\begin{equation*}
|u(y)|e^{\frac{k\phi(x)}{2}-\frac{k\phi(y)}{2}-\delta k}|A^{(N+1)} (x, y, \theta(x, y, \bar y) ) |.
\end{equation*}
So by using the Cauchy-Schwartz inequality, the first integral is bounded by
\begin{equation*}
\|A^{(N+1)}\|_{L^\infty(B^{n}(0,1)\times B^{n}(0,1) \times W_1  )}e^{\frac{k\phi(x)}{2}} \|u\|_{k\phi}O(e^{-\delta k}),
\end{equation*}
where $W_1 \subset W$ is defined by $$W_1=\theta \left ( B^{n}(0,1)\times B^{n}(0,1) \times B^n(0, 1) \right ).$$ 
Similarly, the integrand of the second integral is bounded by some constant times
\begin{equation*}
|u(y)|e^{\frac{k\phi(x)}{2}-\frac{k\phi(y)}{2}}\left|\frac{(D_\theta\cdot A_{N+1} ) (x, y, \theta(x, y, \bar y) )}{k^{N+1}}\right|. 
\end{equation*}
Again using the Cauchy-Schwartz inequality, the second integral is bounded by
\begin{equation*}\label{eq 5.28}
k^{-N-1} \left\| D_\theta\cdot A_{N+1} \right\|_{L^\infty(B^{n}(0,1)\times B^{n}(0,1) \times W_1  )} e^{\frac{k\phi(x)}{2}}\|u\|_{k\phi}.
\end{equation*} 
Finally, \eqref{LocalBergmanKernel2} and \eqref{R}, hence Proposition \ref{LocalBergmanKernel}, follow quickly from the following lemmas.

\begin{lemma}\label{max product of power and exponential}
For any $N \in \mathbb N$,
$$ k e^{-k \delta} \leq \left(\frac{2}{\delta}\right)^{N+2} \frac{ (N+1)!}{k^{N+1}}. $$
\end{lemma}
\begin{proof}
	Define $f(x)=x^{N+2}e^{-\delta x}$ on $[0,\infty)$. It is easy to see that $f$ attains its maximum at $x=\frac{N+2}{\delta}$. Therefore, by using Stirling's formula, we have
	\begin{align*}
		f(x)\leq \left(\frac{N+2}{\delta}\right)^{N+2}e^{-(N+2)}\leq \left(\frac{1}{\delta}\right)^{N+2}(N+2)!.
	\end{align*} 
	The Lemma follows from $N+2\leq 2^{N+2}$.
\end{proof}

\begin{lemma}\label{error term}
We have
\begin{equation}\label{Am}
    \|A_{N}\|_{L^\infty (B^{n}(0,1) \times B^{n}(0,1) \times W_1 )}\leq C^{N}N!^2,
\end{equation}
\begin{equation} \label{Sum}
    \|A^{(N)}\|_{L^\infty (B^{n}(0,1)\times B^{n}(0,1) \times W_1  )}\leq C k + \frac{C^{N}N!^2}{k^N},
\end{equation}
\begin{equation*}
    \|D_\theta\cdot A_{N}\|_{L^\infty (B^{n}(0,1) \times B^{n}(0,1) \times W_1 )}\leq C^{N}N!^2,
\end{equation*}
for a constant $C$ independent of $N$. 
\end{lemma}

\begin{proof}

The idea is to explicitly express $A_m$'s in terms of $b_m$'s, and then use Theorem \ref{MainLemma}. 

First, we observe that since\begin{equation*}
Sa\sim \sum_{i=0}^\infty\frac{(D_\theta\cdot D_y)^i}{i!k^i}\sum_{j=0}^\infty \frac{a_j}{k^j}
\sim \sum_{m=0}^\infty\sum_{i+j=m}\frac{(D_\theta\cdot D_y)^ia_j}{i!k^m},
\end{equation*}
we have
\begin{equation*}
(Sa)_m=\sum_{i+j=m}\frac{(D_\theta\cdot D_y)^ia_j}{i!}.
\end{equation*}
Using the relation \eqref{eq 5.7}, this becomes
\begin{equation*}
(Sa)_m=\sum_{i+j=m}\frac{(D_\theta\cdot D_y)^i(b_j\Delta_0)}{i!} \hspace{12pt} \mbox{ for any }m\geq 1.
\end{equation*}
So by the relation between $Sa$ and $SA$ in \eqref{eq 5.11}, we get
\begin{equation}\label{pqa}
(x-y)\cdot (SA)_{m+1}=\sum_{i+j=m}\frac{(D_\theta\cdot D_y)^i(b_j\Delta_0)}{i!}.
\end{equation}
We emphasize that a formal asymptotic expansion $A$ satisfying the above equation is not unique, and in fact here, we choose the particular $A$ so that $(SA)_0=0$ and $SA$ satisfies
\begin{equation*}
(SA)_{m+1}(x,y,\theta)=\int_0^1D_y\left(\sum_{i+j=m}\frac{(D_\theta\cdot D_y)^i(b_j\Delta_0)}{i!}\right)(x, tx+(1-t)y, \theta)dt \hspace{12pt} \mbox{ for } m\geq 0. 
\end{equation*}
That why \eqref{pqa} holds for this particular $A$ is evident by the fundamental theorem of calculus. 
Given that we have $SA$, by applying the inverse operator $S^{-1}$, we can obtain $A$ as follows
\begin{align*}
\begin{split}
A_m
=&\sum_{i+j=m}\frac{(-D_\theta\cdot D_y)^i(SA)_j}{i!}
\\
=&\sum_{\substack{i+j=m,\\j\geq 1}}\frac{(-D_\theta\cdot D_y)^i}{i!}
\int_0^1D_y\left(\sum_{s+t=j-1}\frac{(D_\theta\cdot D_y)^s(b_t\Delta_0)}{s!}\right)(x, tx+(1-t)y,\theta)dt.
\end{split}
\end{align*}
Now, by using the Cauchy integral formula twice and Theorem \ref{MainLemma}, we obtain
\begin{align*}
\|A_m & \|_{L^\infty (B^{n}(0,1) \times B^{n}(0,1) \times W_1 )}
\\
\leq &
\sum_{\substack{i+j=m,\\j\geq 1}}C^{2i}i!
\left\|\int_0^1D_x\left(\sum_{s+t=j-1}\frac{(D_\theta\cdot D_y)^s(b_t\Delta_0)}{s!}\right)(x, tx+(1-t)y,\theta)dt\right\|_{L^\infty(B^{n}(0,\frac32)\times B^{n}(0,\frac32) \times W_2 )}
\\
\leq &
\sum_{\substack{i+j=m,\\j\geq 1}}C^{2i}i!
\left\|D_y\left(\sum_{s+t=j-1}\frac{(D_\theta\cdot D_y)^s(b_t\Delta_0)}{s!}\right)(x,y,\theta)\right\|_{L^\infty(B^{n}(0,\frac32)\times B^{n}(0,\frac32) \times W_2)}
\\
\leq &
\sum_{\substack{i+j=m,\\j\geq 1}}C^{2i}i!
\sum_{s+t=j-1}C^{2s+1}s!\|b_t\Delta_0\|_{L^\infty(B^{n}(0,\frac52)  \times B^{n}(0,\frac52)  \times W_3)}
\\
\leq &
\sum_{\substack{i+j=m,\\j\geq 1}}C^{2i}i!
\sum_{s+t=j-1}C^{2s+1}s!C^tt!^2
\\
\leq &
C^m m!^2.
\end{align*} In the above series of estimates $W_1 \Subset W_2 \Subset W_3 \Subset W$, where $\Subset$ is our notation for relative compactness. Also, the constants $C$ are renamed to be $C$ again at each line. 

Using the estimate on $A_m$ and \eqref{Dtheta}, we get
\begin{align*}
\|D_\theta\cdot A_m  & \|_{L^\infty(B^{n}(0,1) \times B^{n}(0,1) \times W_1) }  \\
& \leq  \|a_m\|_{L^\infty(B^{n}(0,1) \times B^{n}(0,1) \times W_1)}
+\|(x-y)\cdot A_{m+1}\|_{L^\infty(B^{n}(0,1) \times B^{n}(0,1) \times W_1)}
\leq C^m m!^2,
\end{align*}
for a new $C$. 

To prove \eqref{Sum}, we note that by the definition of $A^{(N)}$  and \eqref{Am}, we have
\begin{equation*}
 \|A^{(N)}\|_{L^\infty (B^{n}(0,1)\times B^{n}(0,1) \times W_1  )}  \leq \frac{C1!^2}{k}+\frac{C^2 2! ^2}{k^2}+\cdots+\frac{C^N N! ^2}{k^N}.
\end{equation*}
To estimate the right hand side, we need to study the function $\frac{C^x (\Gamma(1+x))^2}{k^x}$. By Sterling's formula this is more or less equivalent to studying the function $ \frac{C^xx^{2x}}{e^{2x}k^x}$. To minimize this function we consider its logarithm 
\begin{equation*}
f(x)=\log \frac{C^xx^{2x}}{e^{2x}k^x}=x\log C+2x\log x-2x-x\log k \hspace{12 pt} \mbox{ for } x\in (0,\infty).
\end{equation*}
Since
\begin{equation*}
f'(x)=\log C+2\log x-\log k, 
\end{equation*}
the only critical point is $x_0=(\frac{k}{C})^{\frac{1}{2}}$, and the function is decreasing on the interval $(0, x_0]$ and increasing on the interval $[x_0, \infty)$. Therefore, if we take $N_0=N_0(k)=[(\frac{k}{C})^{\frac{1}{2}}]$, then by using {Stirling's formula} twice
\begin{align*}\sum_{m=1}^N \frac{C^m m! ^2}{k^m} & \leq e^2  \sum_{m=1}^N \frac{C^m m^{2m+1}}{e^{2m}k^m} \\ & \leq e^2 \left ( N_0^2 + N(N-N_0)\frac{C^N N^{2N}}{e^{2N}k^N} \right ) \\ & \leq e^2 \left ({\frac{k}{C}} + \frac{C^N N^{2N+2}}{e^{2N}k^N} \right ) \\ &  \leq C' \left (k +  \frac{C'^N N!^2}{ k^N} \right ). \end{align*}
\end{proof}

\section{From local to global and the proof of Theorem \ref{Main}}\label{Sec Global}
Let $K_k(x,y)$ be the Bergman kernel of $(M,L^k)$. As we noted before, we also write $K_k(x, y)$ for the representation of the Bergman kernel in the local frame $e^k_L \otimes \oo{e^k_L}$ and we denote $K_{k,y}(x):= K_k(x,y)$. In the last section, we constructed the local Bergman kernel of order $N$, which we denoted by $K_{k}^{(N)}(x,y)=K_{k,y}^{(N)}(x)$. In this section, we show that $K_k(x,y)$ is equal to $K_{k}^{(N)}(x,y)$ up to order $k^{-N}$ when $x,y$ are sufficiently close to each other. Moreover, we will give a precise upper bound for the error term.
\begin{prop}\label{LocalGlobal} There exists $ \delta >0$ such that whenever $d(x,y) < \delta$, we have
\begin{equation} \label{Cx}
K_k(x,y)=K_k^{(N)}(x,y)+ k^{\frac{3n}{2}} \widetilde{\mathcal R}_{N+1}(\phi, k) e^{\frac{k\phi(x)}{2}+\frac{k\phi(y)}{2}},
\end{equation}
where
\begin{equation}\label{C}
| \widetilde{\mathcal R}_{N+1}(\phi, k)| \leq \frac{C^{N+1}(N+1)!^2}{k^{N+1}}, 
\end{equation}
and the constant $C$ is independent of $N$, $x$, $y$, and $k$. 
\end{prop}

\begin{proof}
We fix $x \in M$ and assume that $\phi$ is analytic in $B^n(x, 3)$. Let $\chi$ be a  smooth cut-off function such that
\begin{equation*}
    \chi(z)=
    \begin{cases}
    1 & z\in B^{n}(x,\frac{1}{2})\\
    0 & z\notin B^{n}(x,\frac{3}{4})
    \end{cases}.
\end{equation*}
We assume $y\in B^n(x,\frac{1}{4})$.
We first observe that
\begin{equation}\label{S}
K_k(y,x)=\left(\chi K_{k,x},K^{(N)}_{k,y}\right)_{k\phi}+ \mathcal S_{N+1} (\phi, k) k^{\frac{3n}{2}}e^{k(\frac{\phi(x)}{2}+\frac{\phi(y)}{2})},
\end{equation}
where $|\mathcal S_{N+1} (\phi, k)|\leq \frac{C^{N+1}(N+1)!^2}{k^{N+1}}$.
This is because, by Proposition \ref{LocalBergmanKernel}, we have
\begin{align*}
K_{k,x}(y)=\left(\chi K_{k,x}, K^{(N)}_{k,y}\right)_{k\phi}+k^n \mathcal S_{N+1} (\phi, k)   e^{\frac{k\phi(y)}{2}}\|K_{k,x}\|_{k\phi},
\end{align*}
and by the reproducing property of Bergman kernel, we have
\begin{equation*}
\|K_{k,x}\|_{k\phi} e^{-k\frac{\phi(x)}{2}} \leq  \left |  \|K_{k,x}\|_{L^2(M,L^k)} \right |_{h^k}= \sqrt{| K_k(x,x)|_{h^k}}\leq Ck^{\frac{n}{2}}.
\end{equation*}
That why $| K_k(x,x)|_{h^k} \leq Ck^n$ follows from the extreme property of the Bergman function and also the sub-mean value inequality. For a simple proof see for example Lemma 4.1 of \cite{HKSX}.

Next, we define\begin{equation} \label{u}
u_{k,y}(z)= \chi(z) K_{k,y}^{(N)}(z)-\left(\chi K^{(N)}_{k,y}, K_{k,z}\right)_{k\phi}.
\end{equation} Our goal is to estimate $|u_{k, y}(x)|$. 
Since $\left(\chi K^{(N)}_{k,y}, K_{k,x}\right)_{k\phi}$ is the Bergman projection of $\chi K_{k,y}^{(N)}$, $u_{k,y}$ is the minimal $L^2$ solution to the equation
\begin{equation*}
    \dbar u=\dbar(\chi K_{k,y}^{(N)}).
\end{equation*}
So by using H\"ormander's $L^2$ estimates \cite{Ho} (see  \cite{B} for an exposition), we have
\begin{equation*}
\|u_{k,y}\|_{L^2}^2\leq \frac{C}{k}\|\dbar \left(\chi K_{k,y}^{(N)}\right)\|_{L^2}^2.
\end{equation*}
As $K_{k,y}^{(N)}(z)$ is holomorphic, we have $\dbar(\chi K_{k,y}^{(N)})(z)=\dbar\chi(z) K_{k,y}^{(N)}(z)$. We recall that by \eqref{KandB}
$${K^{(N)}_{k,y}(z)}=\left(\frac{k}{\pi}\right)^n e^{k\psi(z,\bar{y})}B^{(N)}(z,\bar{y}). $$ Since $\dbar \chi(z)$ is supported in $d(z, x) \geq \frac12$ and since $d(x, y) \leq \frac14$, using
$$ \text{Re} \, \psi(z, \bar y) \leq - \delta  |z -y|^2 + \frac{\phi(z)}{2} + \frac{\phi(y)}{2}, $$ we get
\begin{equation*}
\left| \dbar \left(\chi K_{k,y}^{(N)}(z)\right) \right| ^2 \leq \left(\frac{k}{\pi}\right)^{2n}e^{-k{\delta}+k{\phi(y)}+k \phi(z)}\left\|B^{(N)}\right\|^2_{L^{\infty}(U \times U)},
\end{equation*}
for a new $\delta>0$.
Therefore,
\begin{align*}
    \left\|\dbar \left(\chi K_{k,y}^{(N)}\right)\right\|_{L^2}^2
    \leq&
    Ck^{2n}\int_{B(x,1)} e^{-k\delta+k\phi(y)}\left\|B^{(N)}\right\|_{L^{\infty}(U \times U)}\frac{\omega^n}{n!}
    \\
    \leq&
    C k^{2n} e^{-k\delta+k\phi(y)}\left\|B^{(N)}\right\|^2_{L^{\infty}(U \times U)},
\end{align*}
which implies
\begin{equation*}
    \|u_{k,y}\|_{L^2}^2
    \leq
    C k^{2n}e^{-k\delta+k\phi(y)}\|B^{(N)}\|^2_{L^{\infty}(U \times U)}.
\end{equation*}
Thus, using the sub-mean value inequality in $B(x,\frac{1}{\sqrt{k}})$ (with respect to the standard Euclidean volume form $dV_0$) as performed below
\begin{align*}
    \|u_{k,y}\|^2
    \geq& e^{-k\phi(x)}
    \int_{B(x,\frac{1}{\sqrt{k}})} |u_{k,y}(z)|^2 e^{-k\left(\phi(z)-\phi(x)\right)} \frac{\omega^n}{n!}
    \\
    \geq& \frac{e^{-k\phi(x)}}{C}
    \int_{B(x,\frac{1}{\sqrt{k}})} |u_{k,y}(z)|^2 dV_0
    \\
    \geq &
    \frac{e^{-k\phi(x)}}{Ck^n}|u_{k,y}(x)|^2,
\end{align*}
we get
\begin{equation*}
|u_{k,y}(x)|\leq Ck^{\frac{3n}{2}}e^{-k\delta}e^{k\left(\frac{\phi(x)}{2}+\frac{\phi(y)}{2}\right)}\left\|B^{(N)}\right\|_{L^\infty(U \times U)}.
\end{equation*}
We can estimate $\|B^{(N)}\|_{L^\infty(U \times U)}$ using our Theorem \ref{MainLemma}
\begin{align*}
 \|B^{(N)}\|_{L^\infty ( U \times U)} & \leq 1+ \frac{1}{k} \|b_1\|_{L^\infty ( U\times U)} + \dots \frac{1}{k^N} \|b_N\|_{L^\infty (U\times U)} \\ &  \leq 1+\frac{C1!^2}{k}+\frac{C^2 2! ^2}{k^2}+\cdots+\frac{C^N N! ^2}{k^N} \\
& \leq C\left(k+\frac{C^NN!^2}{k^N}\right) .
\end{align*}
Hence, 
$$|u_{k,y}(x)|\leq Ck^{\frac{3n}{2}}e^{-k\delta}e^{k\left(\frac{\phi(x)}{2}+\frac{\phi(y)}{2}\right)}\left(k+\frac{C^NN!^2}{k^N}\right) \leq k^{\frac{3n}{2}}e^{k\left(\frac{\phi(x)}{2}+\frac{\phi(y)}{2}\right)} \frac{C^{N+1} (N+1)!^2}{k^{N+1}} . $$ Combining this estimate with \eqref{S} and recalling the definition of $u_{k, y}$ in \eqref{u}, we get the result. 

We point out that we have renewed the constant $C$ at each step, but the final constant is independent of $k$ and $N$. We also note that the constant $C$ may depend on the point $x$, however by a simple compactness argument one can see that each such $C$ can be bounded by a uniform constant independent of $x$.
\end{proof}
Now we are ready to prove Theorem \ref{Main} and its corollaries. 
\subsection{Proof of Theorem \ref{Main}}

By Proposition \ref{LocalGlobal}, we just need to show that with $N=N_0-1= [ \sqrt{k/C}\,]-1$ we have\footnote{For convenience, we use $N_0$ for $N_0(k)= {[\sqrt{k / C}]}$. }
$$  k^{\frac{3n}{2}} \mathcal R_{N_0}(\phi, k) e^{\frac{k\phi(x)}{2}+\frac{k\phi(y)}{2}} = e^{k\left(\frac{\phi(x)}{2}+\frac{\phi(y)}{2}\right)}e^{-\delta k^{\frac{1}{2}}}O(1). $$
However, by the same proposition we know that 
$$ | R_{N_0}(\phi, k)| \leq \frac{C^{N_0}N_0!^2}{k^{N_0}}.$$ 
Hence it is enough to show that
$$ \frac{C^{N_0}N_0!^2}{k^{N_0}} =  e^{-\delta k^{\frac{1}{2}}} O(1).$$ 
By Stirling's formula
\begin{equation*}
\frac{C^{N_0}N_0!^2}{k^{N_0}}
\leq C' N_0\frac{C^{N_0}{N_0}^{2N_0}}{e^{2N_0}k^{N_0}}
 \leq C' N_0e^{-2N_0}
\leq C''k^{\frac{1}{2}}e^{- \frac{2}{\sqrt{C}} k^{\frac{1}{2}}}.
\end{equation*}
Since, $$k^{\frac{3n}{2}+\frac12} e^{- \frac{2}{\sqrt{C}} k^{\frac{1}{2}}} \leq C''' e^{- \frac{1}{\sqrt{C}} k^{\frac{1}{2}}},$$ $\delta =\frac{1}{\sqrt{C}}$ would do the job. The estimate on the derivatives follow immediately by using the Cauchy integral formula over the boundary of a polydisc with radius $\frac{1}{k}$, since the kernels $K_k(x,y)$ and $K_k^{(N)}(x,y)$ are both holomorphic in $x$ and $\bar{y}$  

\subsection{Proof of Corollary \ref{complete asymptotics}}
\begin{proof}
	By Theorem \ref{Main}, uniformly for any $x,y\in U$, we have 
	\begin{equation*}
	K_k(x,y)=e^{k\psi(x,\bar y)}\frac{k^n}{\pi^n}\left (1+\sum _{j=1}^{ N_0(k)-1}\frac{b_j(x, \bar y)}{k^j}\right )+e^{k\left(\frac{\phi(x)}{2}+\frac{\phi(y)}{2}\right)}e^{-\delta k^{\frac{1}{2}}}O(1).
	\end{equation*} 
For any given positive integer $N$, we rewrite the above formula as follows.
	\begin{equation*}
	K_k(x,y)=e^{-k\psi(x,\bar{y})}\frac{k^n}{\pi^n}\left(1+\sum_{j=1}^{N-1}\frac{b_j(x,\bar{y})}{k^j}+\sum_{j=N}^{N_0(k)-1}\frac{b_j(x,\bar{y})}{k^j}+e^{\frac{k}{2}\left(\phi(x)+\phi(y)-2\psi(x,\bar{y})\right)}e^{-\delta k^{\frac{1}{2}}}O(1)\right).
	\end{equation*}
	Our first observation is that, 
	\begin{align*}
	\left|e^{\frac{k}{2}\left(\phi(x)+\phi(y)-2\psi(s,\bar{y})\right)}e^{-\delta k^{\frac{1}{2}}}\right|=e^{\frac{k}{2}D(x,y)-\delta k^{\frac{1}{2}}}
	\leq e^{-\frac{3}{4}\delta k^{\frac{1}{2}}}.
	\end{align*}
	Now we estimate the term $\sum_{j=N}^{N_0-1}\frac{b_j(x,\bar{y})}{k^j}$. By Stirling's formula, we have
	\begin{align*}
	\left|\frac{b_j(x,\bar{y})}{k^j}\right|\leq \frac{C^jj!^2}{k^j}\leq C'j\frac{C^jj^{2j}}{e^{2j}k^j}.
	\end{align*}
	Since $\frac{C^jj^{2j}}{e^{2j}k^j}$ is monotonically decreasing for $1\leq j\leq N_0(k)-1$ (with the help of Stirling's formula once more), we get
	\begin{align*}
	\left|\sum_{j=N}^{N_0-1}\frac{b_j(x,\bar{y})}{k^j}\right|
	\leq& \frac{C^NN!}{k^N}+\sum_{j=N+1}^{N_0-1}C'j\frac{C^jj^{2j}}{e^{2j}k^j}
	\\
	\leq &\frac{C^NN!}{k^N}+C'N_0^2\frac{C^{N+1}(N+1)^{2(N+1)}}{e^{2(N+1)}k^{N+1}}
	\\
	\leq &\frac{C''^{N}N!^2}{k^N}.
	\end{align*}
	Therefore,
	\begin{align*}
	K_k(x,y)
	=&e^{-k\psi(x,\bar{y})}\frac{k^n}{\pi^n}\left(1+\sum_{j=1}^{N-1}\frac{b_j(x,\bar{y})}{k^j}+\frac{C''^NN!^2}{k^N}+e^{-\frac{3}{4}\delta k^{\frac{1}{2}}}O(1)\right).
	\end{align*}
	But by a similar argument as in Lemma \ref{max product of power and exponential},
	\begin{align*}
	e^{-\frac{3}{4}\delta k^{\frac{1}{2}}}\leq \left(\frac{4}{3\delta}\right)^{2N}\frac{(2N)!}{k^N}
	\leq \left(\frac{8}{3\delta}\right)^{2N}\frac{N!^2}{k^N},
	\end{align*}
and the result follows. 
\end{proof}

\subsection{Proof of Corollary \ref{Cor1}}
By Theorem \ref{Main}, we have
\begin{equation*}
K_k(x,y)=e^{k\psi(x, \bar y)}\frac{k^n}{\pi^n}\big(1+\sum _{j=1}^{N_0-1}\frac{b_j(x,\bar y)}{k^j}\big)+e^{\frac{k\phi(x)}{2}+\frac{k\phi(y)}{2}}e^{-\delta k^{\frac{1}{2}}} O(1).
\end{equation*}
Recall that $D(x,y)=\phi(x)+\phi(y)-\psi(x,\bar y)-\psi(y, \bar x)$.
Then
\begin{equation*}
\log|K_k(x,y)|_{h^k}=-\frac{kD(x,y)}{2}+n\log k -n\log \pi +\log\left|1+\sum _{j=1}^{N_0-1}\frac{b_j(x, \bar y)}{k^j}+e^{\frac{Q(x,y)}{2}k-\delta k^{1/2}}O(1)\right|,
\end{equation*}
where $Q(x,y)=\phi(x)+\phi(y)-2\psi(x,\bar{y})$.
So it is sufficient to prove
$$
\log\left|1+\sum _{j=1}^{N_0-1}\frac{b_j(x,\bar{y})}{k^j}+e^{\frac{Q(x,y)}{2}k-\delta k^{1/2}}O(1)\right|
=\log\left(1+O \left ( \frac{1}{k} \right )\right). $$
To do this we note that by our assumption $D(x, y) \leq  \frac{1}{2} \delta k^{-\frac12}$, hence
\begin{equation*}
\left|e^{\frac{Q(x,y)}{2}k-\delta k^{1/2}}\right|=
e^{\frac{D(x,y)}{2}k-\delta k^{1/2}}
\leq e^{-\frac{3\delta}{4}k^{1/2}}.
\end{equation*}
It remains to show that
\begin{equation*}
\left|\sum _{j=1}^{N_0-1}\frac{b_j(x,\bar{y})}{k^j}\right|=O \left (\frac{1}{k} \right ).
\end{equation*}
By the estimates on $b_j(x,\bar{y})$ from Theorem \ref{MainLemma} and Stirling's formula, we have
$$\frac{|b_j(x,\bar{y})|}{k^j}\leq \frac{C^jj!^2}{k^j}\leq C'j\frac{C^j j^{2j}}{e^{2j}k^j}.$$
As shown in Lemma \ref{error term}, the function $f(x)=\log \frac{C^xx^{2x}}{e^{2x}k^x}$ is decreasing on the interval $(0,(\frac{k}{C})^{\frac{1}{2}}]$, thus for $j\in [2,N_0-1]$,
$$\frac{|b_j(x,\bar{y})|}{k^j}\leq C'(N_0-1)\frac{C^2 2^{4}}{e^{4}k^2}\leq C'C^2\frac{N_0}{k^2}.$$
Therefore, 
\begin{equation*}
\left|\sum _{j=1}^{N_0-1}\frac{b_j(x,\bar{y})}{k^j}\right|\leq \frac{C}{k}+C'C^2\frac{N_0^2}{k^2}\leq \frac{C}{k}+C'C\frac{1}{k}=O \left (\frac{1}{k} \right ).
\end{equation*}

\section{Estimates on Bergman Kernel Coefficients}\label{Sec ProofofMainLemma}
As before, we assume the \k metric is analytic in a neighborhood $V$ of $p$. We will estimate the growth rate of the Bergman kernel coefficients $b_m(x,z)$ as $m\rightarrow \infty$ for $x,z$ in some open set $U\subset V$ containing $p$. Our goal is to prove Theorem \ref{MainLemma}. 

The key ingredient for the proof is the following recursive formula\footnote{We discussed its proof in \eqref{Recursive1}. } on $b_m(x,z)$ established in  \cite{BBS}.
\begin{equation}\label{Recursive}
	b_m(x,z(x,x,\theta))=-\sum_{l=1}^m\frac{(D_y\cdot D_\theta)^l}{l!}\big(b_{m-l}\left(x,z(x,y,\theta)\right)\Delta_0(x,y,\theta)\big)|_{y=x}.
\end{equation}
We will break the proof of Theorem \ref{MainLemma} into two steps. The first step is to derive from the recursive formula \eqref{Recursive}, a recursive inequality on $\|D^\xi_zb_m(x,z)\|_{L^\infty(U\times U)}$ for any multi-index $\xi \in (\mathbb{Z}^{\geq 0})^n$. The second step is to estimate $\|D^\xi_zb_m(x,z)\|_{L^\infty(U\times U)}$ by induction. 

In the following we shall use the following standard notations for multi-indicies. 

\begin{itemize}
\item $ \mathbbm 1 = (1, 1, \cdots, 1)$.
\item $|\alpha|=\alpha_1+\alpha_2+\cdots+\alpha_n$.
\item $ \alpha \leq \beta $ if $\alpha_1\leq \beta_1, \alpha_2\leq \beta_2,\cdots,\alpha_n\leq \beta_n$.
\item $\alpha< \beta$ if $ \alpha \leq \beta$ and $\alpha \ne \beta$.
\item $\binom{\alpha}{\beta}=\binom{\alpha_1}{\beta_1}\binom{\alpha_2}{\beta_2}\cdots \binom{\alpha_n}{\beta_n}$.
\item $ \alpha !=\alpha_1!\alpha_2!\cdots \alpha_n!$.
\item $\binom{l}{\delta_1,\delta_2,\cdots,\delta_n}=\frac{l!}{\delta_1!\delta_2!\cdots\delta_n!}$ for any non-negative integer $l$ and multi-index $\delta\geq 0$ such that $|\delta|=l$.
\item $D_y^\alpha=D_{y_1}^{\alpha_1}D_{y_2}^{\alpha_2}\cdots D_{y_n}^{\alpha_n}$.
\end{itemize}

\begin{lemma} \label{bmLemmaAnalytic}
Suppose $\phi$ is real analytic in some neighborhood $V$ of $p$. Then there exist some positive constant $C\geq 1$ and an open set $U\subset V$ containing $p$, such that for any non-negative integer $m$ and multi-index $\xi\geq 0$
\begin{align}\label{bmAnalytic}
	b_{m,\xi}
	\leq\sum_{l=1}^m
	\sum_{|\delta|=l}\delta!
	\sum_{\alpha,\beta\leq \delta}
	\sum_{ |\gamma| \leq |\alpha+\beta|}
	\sum_{\xi_0\leq \xi}
	\frac{b_{m-l,\gamma+\xi_0}}{\gamma!}\frac{\xi!}{\xi_0!}
	\binom{\alpha+|\gamma|\mathbbm{1}}{|\gamma|\mathbbm{1}}\binom{\beta+|\gamma|\mathbbm{1}}{|\gamma|\mathbbm{1}}C^{|\xi-\xi_0|+2|\delta|+|\gamma|+1},
\end{align}
where 
\begin{equation*}
b_{m,\xi}:=\|D_z^\xi b_m(x,z)\|_{L^\infty(U\times U)}.
\end{equation*}
\end{lemma}
\begin{proof}
We first work on $(D_y\cdot D_\theta)^l\Big(b_{m-l}\left(x,z\left(x,y,\theta\right)\right)\Delta_0\left(x,y,\theta\right)\Big)$. We expand the operator $(D_y\cdot D_\theta)^l$ and obtain
\begin{align}\label{DyDtheta}
\begin{split}
	(D_y\cdot D_\theta)^l  & \left(b_{m-l}\left(x,z\left(x,y,\theta\right)\right)\Delta_0\left(x,y,\theta\right)\right)\\
	=&\sum_{|\delta|=l}\binom{l}{\delta_1,\delta_2,\cdots,\delta_n}
	D_y^\delta D_\theta^\delta
	\big(b_{m-l}(x,z(x,y,\theta))\Delta_0(x,y,\theta)\big)\\
	=&\sum_{|\delta|=l}\binom{l}{\delta_1,\delta_2,\cdots,\delta_n}
	\sum_{\alpha,\beta\leq \delta}\binom{\delta}{\alpha}\binom{\delta}{\beta}D_y^\alpha D_\theta^\beta
	\big(b_{m-l}(x,z(x,y,\theta))\big)D_y^{\delta-\alpha}D_\theta^{\delta-\beta}\Delta_0
\end{split}
\end{align}
To compute $D_y^\alpha D_\theta^\beta
	\big(b_{m-l}(x,z(x,y,\theta))\big)$, or equivalently (up to a factor $\alpha ! \beta !$) the Taylor coefficients of $b_{m-l}(x,z(x,y,\theta))$ at arbitrary points, we use the following more general strategy by means of formal Taylor expansions. Assume
\begin{align*}
	&f(z)\sim \sum_{\gamma\geq 0} a_\gamma(z-z_0)^\gamma,
	\\
	&z(y,\theta)\sim \sum _{\alpha,\beta\geq 0} a_{\alpha\beta} (y-y_0)^\alpha(\theta-\theta_0)^\beta,
\end{align*}
where $a_{\alpha\beta}=(a_{\alpha\beta}^1,a_{\alpha\beta}^2,\cdots,a_{\alpha\beta}^n)$ is a vector and $z(y_0,\theta_0)=z_0$.
Then we calculate the Taylor series of the composition $f(z(y,\theta))$ as follows.
\begin{align*}
	f(z(y,\theta))
	\sim& \sum_{\gamma\geq 0} a_\gamma\left(\sum _{\alpha,\beta\geq 0} a_{\alpha\beta} (y-y_0)^\alpha(\theta-\theta_0)^\beta-z_0\right)^\gamma
	\\
	=&\sum_{\gamma\geq 0} a_\gamma\left(\sum _{\alpha+\beta> 0} a_{\alpha\beta} (y-y_0)^\alpha(\theta-\theta_0)^\beta\right)^\gamma
	\\
	=&\sum_{\gamma\geq 0} a_\gamma\left(\sum _{\alpha+\beta> 0} a_{\alpha\beta}^1 (y-y_0)^\alpha(\theta-\theta_0)^\beta\right)^{\gamma_1}
	\cdots
	\left(\sum _{\alpha+\beta> 0}a_{\alpha\beta}^n (y-y_0)^\alpha(\theta-\theta_0)^\beta\right)^{\gamma_n} .
	\end{align*}
	By taking advantage of the following set of indices for convenience,
	$$I(\alpha,\beta,i)= \left \{ \{ \alpha_{ij} \}_{1 \leq j \leq \gamma_i},  \{ \beta_{ij} \}_{1 \leq j \leq \gamma_i}: \quad  \alpha_{ij}+\beta_{ij}>0   \right \},$$
we get
	\begin{align*}
	f(z(y,\theta))=&\sum_{\gamma\geq 0} a_\gamma
	\sum _{I(\alpha,\beta,1)} a^1_{\alpha_{11}\beta_{11}}a^1_{\alpha_{12}\beta_{12}}\cdots
	a^1_{\alpha_{1\gamma_1}\beta_{1\gamma_1}} (y-y_0)^{\alpha_{11}+\alpha_{12}+\cdots+\alpha_{1\gamma_1}}(\theta-\theta_0)^{\beta_{11}+\beta_{12}+\cdots+\beta_{1\gamma_1}}
	\\
	&\cdots
	\sum _{I(\alpha,\beta,n)} a^n_{\alpha_{n1}\beta_{n1}}a^n_{\alpha_{n2}\beta_{n2}}\cdots
	a^n_{\alpha_{n\gamma_n}\beta_{n\gamma_n}} (y-y_0)^{\alpha_{n1}+\alpha_{n2}+\cdots+\alpha_{n\gamma_n}}(\theta-\theta_0)^{\beta_{n1}+\beta_{n2}+\cdots+\beta_{n\gamma_n}}
	\\
	=&a_0+\sum_{\gamma>0} a_\gamma
	\sum _{I(\alpha,\beta,1)}
	\cdots
	\sum _{I(\alpha,\beta,n)}
	a^1_{\alpha_{11}\beta_{11}}a^1_{\alpha_{12}\beta_{12}}\cdots
	a^1_{\alpha_{1\gamma_1}\beta_{1\gamma_1}}\cdots
	a^n_{\alpha_{n1}\beta_{n1}}a^n_{\alpha_{n2}\beta_{n2}}\cdots
	a^n_{\alpha_{n\gamma_n}\beta_{n\gamma_n}}
	\\
	&   \qquad \cdot (y-y_0)^{\alpha_{11}+\alpha_{12}+\cdots\alpha_{n\gamma_n}}
	(\theta-\theta_0)^{\beta_{11}+\beta_{12}+\cdots+\beta_{n\gamma_n}}
	\\
	=&a_0+\sum_{\alpha+\beta> 0}
	\sum_{1\leq |\gamma| \leq |\alpha+\beta|}
	a_\gamma
	\sum_{A(\alpha,\beta,\gamma)}\prod_{i,j} 
	a^i_{\alpha_{ij}\beta_{ij}}(y-y_0)^{\alpha}(\theta-\theta_0)^{\beta}.
	\end{align*}
Here, the last sum runs over the set $A_{\alpha \beta \gamma}$ defined by\footnote{By our notation, $\alpha_{ij}$ and $\beta_{ij}$ are vectors that have nothing to do with $\alpha$ and $\beta$. }
$$A_{\alpha \beta \gamma} = \left \{ \{ \alpha_{ij} \}_{1\leq i \leq n, 1 \leq j \leq \gamma_i},  \{ \beta_{ij} \}_{1\leq i \leq n, 1 \leq j \leq \gamma_i}: \quad  \begin{array}{ll} \sum_{1\leq i \leq n, 1 \leq j \leq \gamma_i} \alpha_{ij}=\alpha, \\ \sum_{1\leq i \leq n, 1 \leq j \leq \gamma_i} \beta_{ij}=\beta,  \\ \alpha_{ij}+\beta_{ij}>0  \end{array}  \right \}.$$
Therefore, the coefficient of $(y-y_0)^{\alpha}(\theta-\theta_0)^{\beta}$ is
	\begin{equation*}
   \sum_{1\leq |\gamma| \leq |\alpha+\beta|}
   a_\gamma
   \sum_{A(\alpha,\beta,\gamma)}\prod_{i,j} 
   a^i_{\alpha_{ij}\beta_{ij}}
	\end{equation*}
when $\alpha+\beta>0$, and is $a_0$ when $\alpha=\beta=0$.

Applying this formula to $b_{m-l}(x,z(x,y,\theta))$ and plugging it into \eqref{DyDtheta}, we get
\begin{align}\label{DDl}
\begin{split}
	(D_y\cdot D_\theta)^l  & \left(b_{m-l}\left(x,z\left(x,y,\theta\right)\right)\Delta_0\left(x,y,\theta\right)\right) \\
 &= b_{m-l}(x,z(x,y,\theta))(D_y\cdot D_\theta)^l\Delta_0 
	\\ & +  \sum_{|\delta|=l}\binom{l}{\delta_1,\delta_2,\cdots,\delta_n}
	\sum_{\substack{\alpha,\beta\leq \delta\\ \alpha+\beta>0}}\binom{\delta}{\alpha}\binom{\delta}{\beta}\alpha!\beta!
	\sum_{1\leq |\gamma| \leq |\alpha+\beta|}
	\frac{D_z^\gamma b_{m-l}(x,z)}{\gamma!}
	\\&\cdot
	\sum_{A_{\alpha \beta \gamma}} \prod_{i,j} 
	\frac{D_y^{\alpha_{ij}}D_\theta^{\beta_{ij}}z_i}{\alpha_{ij}!\beta_{ij}!}\cdot
	D_y^{\delta-\alpha}D_\theta^{\delta-\beta}\Delta_0.
\end{split}
\end{align}
We now substitute \eqref{DDl} into equation \eqref{Recursive} and obtain
\begin{align*}
\begin{split}
	b_m&(x,z(x,x,\theta))
	\\
	=&-\sum_{l=1}^m\frac{1}{l!}
	\Bigg(b_{m-l}(x,z(x,x,\theta))(D_y\cdot D_\theta)^l\Delta_0 (x,x,\theta)
	+ \sum_{|\delta|=l}\sum_{\substack{\alpha,\beta\leq \delta\\ \alpha+\beta>0}}\binom{l}{\delta_1,\delta_2,\cdots,\delta_n}
		\binom{\delta}{\alpha}\binom{\delta}{\beta}\alpha!\beta!
	\\&
	\cdot
	\sum_{1\leq |\gamma| \leq |\alpha+\beta|}
	\frac{D_z^\gamma b_{m-l}(x,z(x,x,\theta))}{\gamma!}
	\sum_{A_{\alpha \beta \gamma}} \prod_{i,j} 
	\frac{D_y^{\alpha_{ij}}D_\theta^{\beta_{ij}}z_i}{\alpha_{ij}!\beta_{ij}!}(x,x,\theta)
	D_y^{\delta-\alpha}D_\theta^{\delta-\beta}\Delta_0(x,x,\theta)\Bigg).
\end{split}
\end{align*}
The correspondence $(x,x,z)\leftrightarrow (x,x,\theta=\psi_x(x,z))$, turns this into
\begin{align*}
\begin{split}
	b_m&(x,z)
	\\
	=&-\sum_{l=1}^m\frac{1}{l!}
	\Bigg(b_{m-l}(x,z)(D_y\cdot D_\theta)^l\Delta_0 (x,x,\psi_x(x,z))
	+ \sum_{|\delta|=l}\sum_{\substack{\alpha,\beta\leq \delta\\ \alpha+\beta>0}}\binom{l}{\delta_1,\delta_2,\cdots,\delta_n}
	\binom{\delta}{\alpha}\binom{\delta}{\beta}\alpha!\beta!
	\\&
	\cdot
	\sum_{1\leq |\gamma| \leq |\alpha+\beta|}
	\frac{D_z^\gamma b_{m-l}(x,z)}{\gamma!}
	\sum_{A_{\alpha \beta \gamma}} \prod_{i,j} 
	\frac{D_y^{\alpha_{ij}}D_\theta^{\beta_{ij}}z_i}{\alpha_{ij}!\beta_{ij}!}(x,x,\psi_x(x,z))
	D_y^{\delta-\alpha}D_\theta^{\delta-\beta}\Delta_0(x,x,\psi_x(x,z))\Bigg).
\end{split}
\end{align*}
Note that in this recursive formula, the coefficients $b_m$ depend on not only the previous coefficients $b_{m-l}$, but also the derivatives of $b_{m-l}$. Hence, we need to include $D_z^\xi b_m$ in our inductive argument. To do this we apply $D_z^\xi$ on both sides and obtain a recursive formula for the derivatives of $b_m$.
\begin{align}\label{bmrecursive}
\begin{split}
D_z^\xi b_m(x,z)
=&-\sum_{l=1}^m\frac{1}{l!}\sum_{\xi_0\leq \xi}\binom{\xi}{\xi_0}
\Bigg(D_z^{\xi_0}b_{m-l}(x,z)D_z^{\xi-\xi_0}\left((D_y\cdot D_\theta)^l\Delta_0 (x,x,\psi_x(x,z))\right)
\\&
+ \sum_{|\delta|=l}\sum_{\substack{\alpha,\beta\leq \delta\\ \alpha+\beta>0}}\binom{l}{\delta_1,\delta_2,\cdots,\delta_n} \delta!^2
\sum_{1\leq |\gamma| \leq |\alpha+\beta|}
\frac{D_z^{\gamma+\xi_0} b_{m-l}(x,z)}{\gamma!}
\\&
\cdot
D_z^{\xi-\xi_0}\left(\sum_{A_{\alpha \beta \gamma}} \prod_{i,j} 
\frac{D_y^{\alpha_{ij}}D_\theta^{\beta_{ij}}z_i}{\alpha_{ij}!\beta_{ij}!}(x,x,\psi_x(x,z))
\frac{D_y^{\delta-\alpha}D_\theta^{\delta-\beta}\Delta_0}{(\delta-\alpha)!(\delta-\beta)!}(x,x,\psi_x(x,z))\right)\Bigg).
\end{split}
\end{align}

As $z(x,y,\psi_x(x,z)), \Delta_0(x,y,\psi_x(x,z))$ are holomorphic, by the Cauchy integral formula, there exists a fixed neighborhood $U$ such that for any $x,z\in U$, we have
	\begin{align*}
	& \left|	D_z^{\xi-\xi_0}  \left(\sum_{A_{\alpha \beta \gamma}}  \prod_{i,j} 
	\frac{D_y^{\alpha_{ij}}D_\theta^{\beta_{ij}}z_i}{\alpha_{ij}!\beta_{ij}!}(x,x,\psi_x(x,z))
	\frac{D_y^{\delta-\alpha}D_\theta^{\delta-\beta}\Delta_0}{(\delta-\alpha)!(\delta-\beta)!}(x,x,\psi_x(x,z))\right)\right|
	\\
	& \qquad \qquad \leq
	\sum_{A(\alpha,\beta,\gamma)} C^{|\xi-\xi_0|+2|\delta|+|\gamma|+1}(\xi-\xi_0)!
	\\
	& \qquad  \qquad \leq\binom{\alpha+|\gamma|\mathbbm{1}}{|\gamma|\mathbbm{1}}\binom{\beta+|\gamma|\mathbbm{1}}{|\gamma|\mathbbm{1}}C^{|\xi-\xi_0|+2|\delta|+|\gamma|+1}(\xi-\xi_0)!.
	\end{align*}
Similarly, 
\begin{align*}
	\left|D_z^{\xi-\xi_0}(D_y\cdot D_\theta)^l\Delta_0 (x,x,\psi_x(x,z))\right|
	\leq \sum_{|\delta|=l}\binom{l}{\delta_1,\delta_2,\cdots,\delta_n}C^{|\xi-\xi_0|+2|\delta|+1}\delta!^2(\xi-\xi_0)!.
\end{align*}
Recall that 
\begin{equation*}
b_{m,\xi}=\|D_z^\xi b_m(x,z)\|_{L^\infty(U\times U)}.
\end{equation*}
Then \eqref{bmrecursive} implies the following inequality 	\begin{align*}
	b_{m,\xi}
            & \leq \sum_{l=1}^m
	\sum_{|\delta|=l}\delta!
         \sum_{\xi_0\leq \xi}\binom{\xi}{\xi_0}
	 b_{m-l,\xi_0}
	C^{|\xi-\xi_0|+2|\delta|+1}(\xi-\xi_0)! \\
	+&\sum_{l=1}^m
	\sum_{|\delta|=l}\delta!
	\sum_{\substack{\alpha,\beta\leq \delta \\ \alpha+ \beta >0}}
	\sum_{1\leq |\gamma| \leq |\alpha+\beta|}
	\sum_{\xi_0\leq \xi}\binom{\xi}{\xi_0}
	\frac{b_{m-l,\gamma+\xi_0}}{\gamma!}
	\binom{\alpha+|\gamma|\mathbbm{1}}{|\gamma|\mathbbm{1}}\binom{\beta+|\gamma|\mathbbm{1}}{|\gamma|\mathbbm{1}}C^{|\xi-\xi_0|+2|\delta|+|\gamma|+1}(\xi-\xi_0)! \\
& \leq \sum_{l=1}^m
	\sum_{|\delta|=l}\delta!
	\sum_{\substack{\alpha,\beta\leq \delta}}
	\sum_{|\gamma| \leq |\alpha+\beta|}
	\sum_{\xi_0\leq \xi}\frac{\xi!}{\xi_0!}
	\frac{b_{m-l,\gamma+\xi_0}}{\gamma!}
	\binom{\alpha+|\gamma|\mathbbm{1}}{|\gamma|\mathbbm{1}}\binom{\beta+|\gamma|\mathbbm{1}}{|\gamma|\mathbbm{1}}C^{|\xi-\xi_0|+2|\delta|+|\gamma|+1}
\end{align*}
Thus Lemma \ref{bmLemmaAnalytic} follows. Next we use this lemma to prove Theorem \ref{MainLemma}.
\end{proof}

\subsection{Proof of Theorem \ref{MainLemma}}
For convenience we define
\begin{equation}\label{am}
a_{m,\xi}=\frac{b_{m,\xi}}{(2m+1)!\xi!}.
\end{equation}
Then by Lemma \ref{bmLemmaAnalytic}
\begin{align}\label{amrecursive}
\begin{split}
	a_{m,\xi}
	\leq \sum_{l=1}^m
	\sum_{|\delta|=l}
	\sum_{\substack{\alpha,\beta\leq \delta \\ |\gamma| \leq |\alpha+\beta| \\ \xi_0\leq \xi }}
	a_{m-l,\gamma+\xi_0}\binom{\gamma+\xi_0}{\xi_0}
	\frac{1}{\binom{2m+1}{2l}}\frac{\delta!}{(2l)!}
	\binom{\alpha+|\gamma|\mathbbm{1}}{|\gamma|\mathbbm{1}}\binom{\beta+|\gamma|\mathbbm{1}}{|\gamma|\mathbbm{1}}C^{|\xi-\xi_0|+2|\delta|+|\gamma|+1}.
\end{split}
\end{align}
Since $b_0(x,z)=1$, we have
\begin{equation}\label{a0}
a_{0,\xi}=
\begin{cases}
1 & \xi=(0,0,\cdots,0),\\
0 & \mbox{otherwise}.
\end{cases}
\end{equation}
We will argue by induction on $m$ and prove that for any integer $m\geq 0$ and multi-index $\xi\geq 0$,
\begin{equation}\label{amupperbound}
a_{m,\xi}\leq\binom{2m+|\xi|}{|\xi|} A^{m}(2C)^{|\xi|},
\end{equation}
where $C$ is the same constant which appears on the right hand side of \eqref{amrecursive} and $A$ is a bigger constant to be selected later. Obviously \eqref{a0} and the fact that  $C \geq 1$ imply that \eqref{amupperbound} holds for $m=0$ and any $\xi\geq 0$. Assume that \eqref{amupperbound} holds up to $m-1$ and we proceed to $m$. By \eqref{amrecursive} and because $\frac{\delta!}{(2l)!}\leq 1$, we have
\begin{align*}
	a_{m,\xi}
	\leq&\sum_{l=1}^m
	\sum_{|\delta|=l}
	\sum_{\alpha,\beta\leq \delta}
	\sum_{ |\gamma| \leq |\alpha+\beta|}
	\sum_{\xi_0\leq \xi}
	A^{m-l}\binom{2m-2l+|\gamma+\xi_0|}{|\gamma+\xi_0|}
	\binom{\gamma+\xi_0}{\xi_0}
	\frac{1}{\binom{2m+1}{2l}}
	\\&
	\cdot \binom{\alpha+|\gamma| \mathbbm{1}}{|\gamma|\mathbbm{1}}\binom{\beta+|\gamma|\mathbbm{1}}{|\gamma|\mathbbm{1}}2^{|\gamma+\xi_0|}C^{|\xi|+2|\gamma|+2|\delta|+1}
	\\
	\leq&\sum_{l=1}^m
	\sum_{|\delta|=l}
	\sum_{ |\gamma| \leq 2|\delta|}
	\sum_{\xi_0\leq \xi}
	A^{m-l}\binom{2m-2l+|\gamma+\xi_0|}{|\gamma+\xi_0|}
	\binom{\gamma+\xi_0}{\xi_0}
	\frac{1}{\binom{2m+1}{2l}}
	\\&
	 \cdot \sum_{\alpha\leq \delta}\binom{\alpha+|\gamma|\mathbbm{1}}{|\gamma|\mathbbm{1}}
	\sum_{\beta\leq \delta}\binom{\beta+|\gamma| \mathbbm{1}}{|\gamma|\mathbbm{1}}2^{|\gamma+\xi_0|}C^{|\xi|+7l}.
\end{align*}
Due to the fact
\begin{equation*}
\sum_{\alpha\leq \delta}\binom{\alpha+|\gamma|\mathbbm{1}}{|\gamma|\mathbbm{1}}
=\sum_{\beta\leq \delta}\binom{\beta+|\gamma|\mathbbm{1}}{|\gamma|\mathbbm{1}}
=\binom{\delta+(|\gamma|+1)\mathbbm{1}}{(|\gamma|+1)\mathbbm{1}}\leq 2^{|\delta|+n|\gamma|+n},
\end{equation*}
it follows that
\begin{align*}
	a_{m,\xi}
	\leq&\sum_{l=1}^m
	\sum_{|\delta|=l}
	\sum_{|\gamma| \leq 2|\delta|}
	\sum_{\xi_0\leq \xi}
	A^{m-l}\binom{2m-2l+|\gamma+\xi_0|}{|\gamma+\xi_0|}
	\binom{\gamma+\xi_0}{\xi_0}
	\frac{1}{\binom{2m+1}{2l}}
	2^{|\gamma+\xi_0|+2|\delta|+2n|\gamma|+2n}C^{|\xi|+7l}
	\\
	\leq&A^m(2C)^{|\xi|}
	\sum_{l=1}^m
	\sum_{|\delta|=l}
	\sum_{ |\gamma| \leq 2l}
	\sum_{\xi_0\leq \xi}
	2^{|\xi_0|-|\xi|}\left(\frac{2^{6n+4}C^7}{A}\right)^{l}
	\binom{2m-2l+|\gamma+\xi_0|}{|\gamma+\xi_0|}
	\binom{\gamma+\xi_0}{\xi_0}
	\frac{1}{\binom{2m+1}{2l}}.
\end{align*}
Moreover, since
\begin{align*}
\#\{|\delta|=l\}=\binom{l+n-1}{n-1}\leq 2^{l+n-1}\leq 2^{nl},
\end{align*}
we have
\begin{align}\label{am2}
	a_{m,\xi}
	\leq A^m(2C)^{|\xi|}
	\sum_{l=1}^m
	\sum_{ |\gamma| \leq 2l}
	\sum_{\xi_0\leq \xi}
	2^{|\xi_0|-|\xi|}(\frac{2^{7n+4}C^7}{A})^{l}
	\binom{2m-2l+|\gamma+\xi_0|}{|\gamma+\xi_0|}
	\binom{\gamma+\xi_0}{\xi_0}
	\frac{1}{\binom{2m+1}{2l}}.
\end{align}
In the next step we apply the combinatorial inequality
\begin{equation*}
    \binom{\gamma+\xi_0}{\xi_0}\leq \binom{|\gamma+\xi_0|}{|\xi_0|},
\end{equation*}
and the combinatorial identity
\begin{align*}
\binom{2m-2l+|\gamma+\xi_0|}{|\gamma+\xi_0|}\binom{|\gamma+\xi_0|}{|\xi_0|}
=\binom{2m-2l+|\gamma+\xi_0|}{|\xi_0|}\binom{2m-2l+|\gamma|}{2m-2l}.
\end{align*}
Observe that, since $|\gamma|\leq 2l$ and $\xi_0\leq \xi$, we have
\begin{align*}
\binom{2m-2l+|\gamma+\xi_0|}{|\xi_0|}
\leq \binom{2m+|\xi|}{|\xi|}.
\end{align*}
Plugging these into \eqref{am2}, we obtain
\begin{align*}
	a_{m,\xi}
	\leq A^m(2C)^{|\xi|}\binom{2m+|\xi|}{|\xi|}
	\sum_{l=1}^m
	\sum_{ |\gamma| \leq 2l}
	\sum_{\xi_0\leq \xi}
	2^{|\xi_0|-|\xi|}(\frac{2^{7n+4}C^7}{A})^{l}
	\binom{2m-2l+|\gamma|}{2m-2l}
	\frac{1}{\binom{2m+1}{2l}}.
\end{align*}
Again since
\begin{equation*}
    \#\{|\gamma|=k\}=\binom{k+n-1}{n-1}\leq 2^{k+n-1},
\end{equation*}
the sum over $\gamma$ on the right hand side can be estimated as
\begin{align*}
    \sum_{ |\gamma|\leq 2l} \binom{2m-2l+|\gamma|}{2m-2l}
    =\sum_{k=0}^{2l}\sum_{|\gamma|=k}\binom{2m-2l+k}{2m-2l}
    \leq  2^{2l+n-1}\binom{2m+1}{2m-2l+1}.
\end{align*}
Therefore,
\begin{align*}
	a_{m,\xi}
	\leq& A^m(2C)^{|\xi|}\binom{2m+|\xi|}{|\xi|}
	\sum_{l=1}^m
    \left (\frac{2^{8n+5}C^7}{A} \right )^{l}
	\sum_{\xi_0\leq \xi}
	2^{|\xi_0|-|\xi|}
	\\
	\leq& A^m(2C)^{|\xi|}\binom{2m+|\xi|}{|\xi|}
	\sum_{l=1}^m
    2^n \left (\frac{2^{8n+5}C^7}{A}\right )^{l}.
\end{align*}
By taking  $A= 2^{9n+6}C^7$ we surely have $\sum_{l=1}^m2^n \left (\frac{2^{8n+5}C^7}{A} \right )^{l} < 1$, which implies that $a_{m,\xi}\leq A^m(2C)^\xi\binom{2m+|\xi|}{|\xi|}$, hence concluding the induction. Remembering the definition \eqref{am} of $a_m$ in terms of $b_m$, we get
\begin{equation}\label{bmUB}
\|D_z^\xi b_m(x,z)\|_{L^\infty(U\times U)}= (2m+1)!\xi! a_{m,\xi}\leq (64A)^{m+|\xi|}(m!)^2\xi!,
\end{equation}
So Theorem \ref{MainLemma} follows by renaming $64A$ to $C$ and taking $\xi=0$.

\section{Optimality of the upper bounds on Bergman coefficients $b_m$} \label{optimal}In this section we show that although it would be desirable to improve the estimate \eqref{bmUB} to
\begin{equation}\label{bmconjecture}
\|D_z^\xi b_m(x,z)\|_{L^\infty(U\times U)}\leq C^{m+|\xi|} m!\xi!,
\end{equation}
 it is not possible to prove it simply by the recursive inequality \eqref{bmAnalytic}. Here we provide an example which satisfies \eqref{bmAnalytic} while fails \eqref{bmconjecture}. For simplicity, we assume $C=1$ in \eqref{bmAnalytic}. Let us consider the worst case when equality holds in \eqref{bmAnalytic}, i.e.
\begin{equation}\label{bmoptimal}
	b_{m,\xi}
	=\sum_{l=1}^m
	\sum_{|\delta|=l}\delta!
	\sum_{\alpha,\beta\leq \delta}
	\sum_{ |\gamma| \leq |\alpha+\beta|}
	\sum_{\xi_0\leq \xi}
	\frac{b_{m-l,\gamma+\xi_0}}{\gamma!}\frac{\xi!}{\xi_0!}
	\binom{\alpha+|\gamma|\mathbbm{1}}{|\gamma|\mathbbm{1}}\binom{\beta+|\gamma|\mathbbm{1}}{|\gamma|\mathbbm{1}}.
\end{equation}
One can easily check that this recursive equation uniquely defines $\{b_{m, \xi} \}$ given an initial data $\{ b_{0, \xi} \}$. We shall only focus on the terms $b_{m,ke_1}$ where $e_1=(1,0,\cdots,0)$ and show by induction that
\begin{equation}\label{bmLB}
b_{m,ke_1}\geq (2m-2+k)! \hspace{12pt} \mbox{ for any }m\geq 1,k\geq 0.
\end{equation}
First let us check this for $b_{1,ke_1}$. Since in our case
\begin{equation*}
b_{0,\xi}=
\begin{cases}
1 & \xi=0\\
0 & \mbox{otherwise},
\end{cases}
\end{equation*}
by \eqref{bmoptimal}, we have
\begin{equation*}
	b_{1,\xi}
	=\sum_{|\delta|=1}
	\sum_{\alpha,\beta\leq \delta}\xi!\geq \xi!.
\end{equation*}
Therefore \eqref{bmLB} holds for $b_{1,ke_1}$.
Assume that \eqref{bmLB} holds for $b_{1,ke_1},b_{2,ke_1},\cdots b_{m-1,ke_1}$. Then by only considering the terms with $l=|\alpha|=|\beta|=1$ and $\gamma=2e_1$ in \eqref{bmoptimal}, we obtain for $m\geq 2$ 
\begin{align*}
	b_{m,ke_1}
	\geq&
	\sum_{|\delta|=1}
	\sum_{|\alpha|=|\beta|=1}
	\sum_{j=0}^k
	\frac{b_{m-1,(j+2)e_1}}{2!}\frac{k!}{j!}
	\binom{\alpha+2 \cdot \mathbbm{1}}{2 \cdot \mathbbm{1}}\binom{\beta+2 \cdot \mathbbm{1}}{2 \cdot \mathbbm{1}}
	\\
	\geq&
	\sum_{j=0}^k
	(2m-2+j)!\frac{k!}{j!}
	\\
	\geq&
	(2m-2+k)!.
    \end{align*}
Note that if in particular we put $k=0$ into \eqref{bmLB} we get 
$$ b_{m,0} \geq  \left ( \frac14 \right)^m  m!^2, $$ 
which shows that up to an exponential factor $C^m$, $m!^2$ is the best upper bound one can hope to obtain from the recursive inequality \eqref{bmAnalytic}. 

\section{\k Manifolds with local constant holomorphic sectional curvatures}\label{Sec ConstantCurvature}

In this section, we consider \k manifolds with local constant holomorphic sectional curvature and prove Theorem \ref{Constant1} and Corollary \ref{Constant2}. If the \k manifold has constant holomorphic sectional curvatures only near a point $p$, then we have the following properties on $b_m$ near $p$. 
\begin{prop}\label{local constant recursive formula}
	Assume the \k manifold has constant holomorphic sectional curvature $c$ in some neighborhood $V$ containing $p$. Then there exists $U \subset V$ containing $p$  such that for any $x,y\in U$, the Bergman kernel coefficients $b_m(x,\bar{y})$ are all constants that vanish for $m > n$, and are given by the polynomial relation
\begin{align*}
		\sum_{j=0}^n{b_j}k^{n-j}=  \frac{c^n \Gamma (\frac{k}{c}+n+1)}{\Gamma(\frac{k}{c}+1)} = c^n \left (\frac{k}{c}+n \right ) \left (\frac{k}{c}+n-1 \right ) \dots \left (\frac{k}{c}+1 \right ). 
		\end{align*}
In the case $c=0$, the right hand side is understood as the limit when $c \to 0$, which equals $k^n$. 

\end{prop}

To prove this proposition we first prove a lemma that gives a recursive formula for the constants $b_m$. 
\begin{lemma} \label{recursiveCHSC}For all $c$, $b_0=1$ and for $m\geq 1$ we have
	\begin{equation}\label{constant recursive}
	b_m=-\sum_{l=1}^m(-c)^l\frac{(l+n-1)!}{(n-1)!}a_lb_{m-l},
	\end{equation}
	where $\{a_l\}_{l=0}^\infty$ are given by the Taylor expansion
	\begin{equation}\label{definition of a_l}
	e^x\left(\frac{e^x-1}{x}\right)^{n-1}=\sum_{l=0}^\infty a_lx^l.
	\end{equation}
\end{lemma}
\begin{proof}
	When the \k manifold has constant holomorphic sectional curvature $c$ in $V$ containing $p$, it is known that (see for example equation (28) in \cite{Boc}),  in a specific coordinate (Bochner coordinate) at $p$, the \k potential near $p$, say in $U$, can be uniquely written\footnote{Note that we are using a different notation from the one in \cite{Boc}, so that ${\mathbb C \mathbb P}^n$ has holomorphic sectional curvature $1$ instead of $2$. } as 
\begin{equation*}
	\phi(x)=\begin{cases}
	\frac{1}{c}\ln\left(1+c\sum_{i=1}^n|x_i|^2\right) \quad &c\neq 0,
	\\
	\sum_{i=1}^{n}|x_i|^2 \quad &c=0,
	\end{cases}
	\end{equation*}
	where the case $c=0$ can be regarded as the limiting case when $c\rightarrow 0$. We now  simplify the recursive formula \eqref{Recursive} on $b_m$ with the above explicit expression for $\phi$. By polarizing $\phi(x)$, we get $\psi(x,z)=\frac{1}{c}\ln\left(1+cx\cdot z\right)$. Taking partial derivatives, 
	\begin{align*}
	\psi_{x_i}(x,z)=\frac{z_i}{1+cx\cdot z}, & \qquad 
	\psi_{x_iz_j}=\frac{\delta_{ij}}{1+cx\cdot z}-\frac{cz_ix_j}{(1+cx\cdot z)^2}.
	\end{align*}
	The \emph{matrix determinant lemma} tells that if $A$ is an invertible matrix and $u$ and $v$ are column vectors, then
	\begin{align*}
	\det\left(A+uv^\intercal\right)=\left(1+v^\intercal A^{-1}u\right)\det A.
	\end{align*}
	Therefore, 
	\begin{align*}
	\det \psi_{xz}(x,z)=\frac{1}{(1+cx\cdot z)^{n+1}}.
	\end{align*}
	Recalling the definition of $\theta=\theta(x,y,z)$ in \eqref{theta}, we have
	\begin{align*}
	\theta=\frac{\ln(1+cx\cdot z)-\ln(1+cy\cdot z)}{c(x-y)\cdot z}\, z.
	\end{align*}
	Hence,
	\begin{align*}
	\frac{\partial \theta_i}{\partial z_j}(x,y,z)
	=&\frac{\ln(1+cx\cdot z)-\ln(1+cy\cdot z)}{c(x-y)\cdot z}\delta_{ij}+\frac{\frac{cx_j}{1+cx\cdot z}-\frac{cy_j}{1+cy\cdot z}}{c(x-y)\cdot z}z_i
	\\
	&-\frac{\ln(1+cx\cdot z)-\ln(1+cy\cdot z)}{(c(x-y)\cdot z)^2} c(x_j-y_j)z_i.
	\end{align*}
	Using the \emph{matrix determinant lemma} again, we obtain
	\begin{align*}
	\det\theta_z=\frac{1}{(1+cx\cdot z)(1+cy\cdot z)}\left(\frac{\ln(1+cx\cdot z)-\ln(1+cy\cdot z)}{c(x-y)\cdot z}\right)^{n-1}.
	\end{align*}
	Therefore,
	\begin{align*}
	\Delta_0(x,y,z)
	=&\frac{\det\psi_{yz}(y,z)}{\det\theta_z(x,y,z)}\\
	=&\frac{1+cx\cdot z}{1+cy\cdot z}\left(\frac{c(x-y)\cdot z}{1+cy\cdot z}\right)^{n-1} \left(\frac{1}{\ln(1+cx\cdot z)-\ln(1+cy\cdot z)}\right)^{n-1}.
	\end{align*}
	We then use the relation $c\theta\cdot(x-y)=\ln(1+cx\cdot z)-\ln(1+cy\cdot z)$, to change the variables from $(x,y,z)$ to $(x,y,\theta)$ and get
	\begin{align*}
	\Delta_0(x,y,\theta)=e^{c\theta\cdot (x-y)}\left(\frac{e^{c\theta\cdot(x-y)}-1}{c\theta\cdot (x-y)}\right)^{n-1}=\sum_{l=0}^\infty a_l\left(c\theta\cdot (x-y)\right)^l.
	\end{align*}
	Remember that $b_0(x,z)=1$, so if we assume that $b_1(x,z),b_2(x,z)\cdots b_{m-1}(x,z)$ are all constants, then by \eqref{Recursive} we have
	\begin{align*}
	b_m=-\sum_{l=1}^m\frac{(D_y\cdot D_\theta)^l}{l!}\left(b_{m-l}\Delta_0\right) \big|_{y=x}
	=-\sum_{l=1}^mb_{m-l}\frac{(D_y\cdot D_\theta)^l\Delta_0}{l!} \big|_{y=x}.
	\end{align*}
	Now we calculate $\frac{(D_y\cdot D_\theta)^l\Delta_0}{l!} \big|_{y=x}$ as follows. 
	\begin{align*}
	\frac{(D_y\cdot D_\theta)^l\Delta_0}{l!} \big|_{y=x}
	=&\sum_{j=0}^\infty a_j\frac{(D_y\cdot D_\theta)^l}{l!} \left(c\theta\cdot (x-y)\right)^j \big|_{y=x}\\
	=&a_l\frac{(D_y\cdot D_\theta)^l}{l!} \left(c\theta\cdot (x-y)\right)^l \big|_{y=x}\\
	=&\frac{c^la_l}{l!}\sum_{|\xi|=l}\binom{l}{\xi_1,\xi_2,\cdots,\xi_n}D_y^\xi D_\theta^\xi \sum_{|\eta|=l}\binom{l}{\eta_1,\eta_2,\cdots,\eta_n}\theta^\eta(x-y)^\eta \,\, \big|_{y=x}
	\\
	=&\frac{c^la_l}{l!}\sum_{|\xi|=l}\binom{l}{\xi_1,\xi_2,\cdots,\xi_n}^2\xi!^2(-1)^{|\xi|}
	\\
	=&(-c)^la_l\frac{(l+n-1)!}{(n-1)!}.
	\end{align*} 
	Therefore, $b_m=-\sum_{l=1}^m(-c)^l\frac{(l+n-1)!}{(n-1)!}a_lb_{m-l}$ is a constant and the lemma follows.
\end{proof}

\begin{proof}[\textbf{Proof of Proposition \ref{local constant recursive formula}.}]
	
	By Lemma \ref{recursiveCHSC}, it is obvious that if $c=0$ then $b_m=0$ for any $m\geq 1$. When $c\neq 0$, 
	we define $\widetilde{b}_m=\frac{b_m}{c^m}$. Then $\tilde{b}_0=b_0=1$ and $\tilde{b}_m$'s satisfy
	\begin{align*}
	\widetilde{b}_m=-\sum_{l=1}^m(-1)^l\frac{(l+n-1)!}{(n-1)!}a_l\widetilde{b}_{m-l}.
	\end{align*}
	Since $\widetilde{b}_m$ share the same initial data and the same recursive formula as the Bergman coefficients for $c=1$, $\widetilde{b}_m$ are identical to the Bergman coefficients for $c=1$, in particular those of ${\mathbb C \mathbb P}^n$. On the other hand, one can find by direct computation (see for instance, Example 1 in \cite{Lu}) that for ${\mathbb C \mathbb P}^n$, $\widetilde{b}_m=0$ for $m>n$, and $$\sum_{j=0}^n\frac{\widetilde{b}_j}{k^j}=\frac{1}{k^n}\frac{\Gamma(k+n+1)}{\Gamma(k+1)}=\frac{1}{k^n} \left ({k}+n \right ) \left ({k}+n-1 \right ) \dots \left ({k}+1 \right ).$$ 
Then, using this we write
$$\sum_{j=0}^n{b_j}k^{n-j} = \sum_{j=0}^n{c^j \widetilde{b}_j}k^{n-j} = k^n \sum_{j=0}^n{ \widetilde{b}_j}\left(\frac{k}{c}\right)^{-j}=\frac{c^n\Gamma(\frac{k}{c}+n+1)}{\Gamma(\frac{k}{c}+1)}, $$
and the result follows.
\end{proof}

\begin{rmk}
	If the \k manifold has global constant holomorphic sectional curvature $c$, then the fact that the Bergman coefficients are constants and vanish for $m > n$, can be obtained from the Hirzebruch-Riemann-Roch Theorem as we demonstrate below. 
	
	Suppose $M$ has constant holomorphic sectional curvature. Then by Theorem 10 of \cite{Boc}, the curvature $R_{i\bar{j}k\bar{l}}$ can be written in terms of the metric $g_{i\bar{j}}$ as
	\begin{equation*}
	R_{i\bar{j}k\bar{l}}=-c\left(g_{i\bar{j}}g_{k\bar{l}}+g_{i\bar{l}}g_{k\bar{j}}\right).
	\end{equation*} 
	By Theorem 1.1 in \cite{Lu}, each $b_{m}(x,\bar{x})$ is a polynomial of the curvature and its covariant derivatives at $x$. Therefore, each $b_{m}(x,\bar{x})$ is a constant function on $M$. Denote $B_k(x)$ to be the Bergman function on $M$ defined by $B_k(x)=|K_k(x,x)|_{h^k}$. Then 
	$$\dim H^0(M,L^k)=\int_M B_k(x) \frac{\omega^n}{n!}.$$
	By the Kodaira Vanishing theorem and  the Hirzebruch-Riemann-Roch theorem, for sufficiently large $k$,  we have
	\begin{align*}
	\dim H^0(M,L^k)=\chi(M,L^k)=\int_M Ch(L^k)\wedge Td(M),
	\end{align*}
	where $Ch(L^k)$ is the Chern character and $Td(M)$ is the Todd class.
	Note that, since
	\begin{align*} \int_M  Ch(L^k) & \wedge Td(M)\\
	=&\int_M Td(M) +\frac{k}{\pi}\int_M\frac{\omega}{\pi}\wedge Td(M)+\left(\frac{k}{\pi}\right)^2\int_M\frac{\omega^2}{2!}\wedge Td(M)+\cdots +\left(\frac{k}{\pi}\right)^n\int_M\frac{\omega^n}{n!},
	\end{align*}
	the quantity $\dim H^0(M,L^k)$ must be a polynomial in $k$.  On the other hand, by the asymptotic expansion of $K_k(x,x)$ we have
	\begin{align*}
	\int_M B_k(x)\frac{\omega^n}{n!}
	=\left(\frac{k}{\pi}\right)^n\int_M\frac{\omega^n}{n!}\left(b_0(x,\bar{x})+\frac{b_1(x,\bar{x})}{k}+\frac{b_2(x,\bar{x})}{k^2}+\cdots\right).
	\end{align*}
	By comparing the coefficients and knowing that $b_{m}(x,\bar{x})$ are constants,  we conclude that $b_{m}(x,\bar{x})$ vanishes when $m>n$. Note that because $b_m(x,z)$ is the polarization of $b_m(x,\bar{x})$, each $b_m(x,z)$ wherever it is defined, is the same constant as $b_m(x,\bar{x})$. So  $b_m(x,z)=0$ for $m >n$. 
\end{rmk}

Now that in this particular case, we have better estimates on the growth of $b_m(x,z)$, the proofs of Theorem \ref{Constant1} and Corollary \ref{Constant2} follow by using the same argument as in the previous sections. We shall only indicate the differences. 

\subsection{Proof of Theorem \ref{Constant1}}
Since $\|b_m(x,z)\|_{L^\infty(U\times U)}$ is now trivially bounded by $C^mm!$ (in fact it is bounded by a constant $C$) instead of $C^mm!^2$, estimates in Lemma \ref{error term} can be improved to
\begin{align*}
&\|A_{N}\|_{L^\infty (B^{n}(0,1) \times B^{n}(0,1) \times W_1 )}\leq C^{N}N!,\\ 
&\|A^{(N)}\|_{L^\infty (B^{n}(0,1)\times B^{n}(0,1) \times W_1  )}\leq C k^2 + \frac{C^{N}N!}{k^N},\\
&\|D_\theta\cdot A_{N}\|_{L^\infty (B^{n}(0,1) \times B^{n}(0,1) \times W_1 )}\leq C^{N}N!.
\end{align*}
Accordingly, the error term estimates in Proposition \ref{LocalGlobal} and \ref{LocalBergmanKernel} can be both improved to 
\begin{align*}
| \mathcal R_{N+1}(\phi, k)| \leq \frac{C^{N+1}(N+1)!}{k^{N+1}}.
\end{align*}
By choosing $N=[\frac{k}{C}]-1$, we can minimize the error term $\mathcal R_{N+1}(\phi, k)$ to be $e^{-\delta k}O(1)$ and thus obtain the desired results by observing that $[\frac{k}{C}]-1>n$ for sufficiently large $k$.


\section{Transport equations for the amplitude}\label{Sec NewFormula}
In this final section, we present a parametrix for the local Bergman kernel by means of \emph{transport equations}, similar to those for the wave and heat equations, which could be of independent interest. One can also give a proof of Theorem \ref{MainLemma} using this parametrix representation.  

In the following $K_{k}^{(N)}(x,y)=K_{k,y}^{(N)}(x)$ is the local Bergman kernel of order $N$, defined in Proposition \ref{LocalBergmanKernel}. See Section \ref{ReviewBBS} for all other necessary notations. 
\begin{thm}\label{Transport} For all $N \in \mathbb N$, the local Bergman kernel of order $N$ can be expressed as 
$$K_{k}^{(N)}(x,y)= \left(\frac{k}{\pi}\right)^n  e^{k\psi(x,\bar{y})} B^{(N)}(x, \bar y),$$
where $B^{(N)}(x, \bar y)$ is the truncation up to order $k^{-N}$ of the formal expansion
$$B(x, \bar y) = b_0(x, \bar y) + \frac{b_1(x, \bar y)}{k} +  \frac{b_2(x, \bar y)}{k^2} + \dots,$$
given by 
$$B(x,\bar y)=\frac{1+k(x-y)\cdot A(x,y,\theta(x,y,\bar y))+(D_\theta \cdot A)(x,y,\theta(x,y, \bar y))}{\Delta_0(x,y,\theta(x,y, \bar y))},$$
where $$A(x, y, \theta)= \frac{A_1(x, y,\theta)}{k} + \frac{A_2(x, y,\theta)}{k^2} + \dots ,$$ is the formal expansion whose coefficients satisfy the following transport equations
\begin{equation*}
A_1(x,y,\theta)=-\int_0^1 (D_y\Delta_0)(x,tx+(1-t)y,\theta)dt,
\end{equation*}
and for $m \geq 2$
\begin{align}\label{NewRecursive}
\small \begin{split}
A_m(x,y,\theta)  = & -\int_0^1D_y\Big(\Delta_0(x,y,\theta)(D_\theta \cdot A_{m-1})(x,x,\theta(x,x,z(x,y,\theta))) \Big) (x,tx+(1-t)y,\theta)dt\\
& + \int_0^1D_y\Big((D_\theta \cdot A_{m-1})(x,y,\theta)\Big)(x,tx+(1-t)y,\theta)dt. 
\end{split}
\end{align} 
\end{thm}
\begin{proof}
Recall that the asymptotic expansion $a(x,y,\theta)$ satisfies
\begin{equation*}
1+a(x,y,\theta(x,y,z))=B(x,z)\Delta_0(x,y,\theta(x,y,z)).
\end{equation*}
Since $a$ is \textit{negligible}, we have $(Sa)(x,x,\theta)=0$ and hence there exists a formal expansion $C(x,y,\theta)$ such that
\begin{equation*}
Sa(x,y,\theta)=(x-y)\cdot C(x,y,\theta).
\end{equation*}
We can solve for $a$ by taking $S^{-1}$ and obtain
\begin{align*}
a=(x-y)\cdot S^{-1}C+\frac{1}{k}D_\theta\cdot S^{-1}C.
\end{align*}
Then we set $A=\frac{1}{k}S^{-1}C$. Obviously,
\begin{equation*}
a(x,y,\theta)=k(x-y)\cdot A(x,y,\theta)+D_\theta\cdot A(x,y,\theta).
\end{equation*}
Thus $B(x,z)$ is related to $A$ by
\begin{equation}\label{BandA}
B(x,z)=\frac{1+k(x-y)\cdot A(x,y,\theta(x,y,z))+(D_\theta \cdot A)(x,y,\theta(x,y,z))}{\Delta_0(x,y,\theta(x,y,z))}.
\end{equation}
Since $B(x,z)$ is independent of the variable $y$, the right hand side is unchanged when we vary $y$. In particular, we can put $y=x$ and get
\begin{equation*}
\frac{1+k(x-y)\cdot A(x,y,\theta(x,y,z))+(D_\theta \cdot A)(x,y,\theta(x,y,z))}{\Delta_0(x,y,\theta(x,y,z))}
=1+(D_\theta \cdot A)(x,x,\theta(x,x,z)).
\end{equation*} 
Plugging the asymptotic expansion $$A\sim A_0+\frac{A_1}{k}+\frac{A_2}{k^2}+\cdots,$$ and comparing the coefficients, we obtain $A_0=0$ and
\begin{align*}
(x-y)\cdot A_1(x,y,\theta)=\Delta_0(x,y,\theta)-1.
\end{align*}
By the fundamental theorem of calculus, this equation has a particular solution given by
\begin{equation*}
A_1(x,y,\theta)=-\int_0^1 (D_y\Delta_0)(x,tx+(1-t)y,\theta)dt.
\end{equation*}
For $m\geq 2$, we get
\begin{align*}
\begin{split}
(x-y)\cdot A_m & (x,y,\theta(x,y,z))= \\ & \Delta_0(x,y,\theta(x,y,z))(D_\theta \cdot A_{m-1})(x,x,\theta(x,x,z))-(D_\theta \cdot A_{m-1})(x,y,\theta(x,y,z)),
\end{split}
\end{align*}
which by changing variables to $(x,y,\theta)$, turns into
\begin{align*}
\begin{split}
\small (x-y)\cdot &A_m(x,y,\theta) =\Delta_0(x,y,\theta)(D_\theta \cdot A_{m-1})(x,x,\theta(x,x,z(x,y,\theta)))-(D_\theta \cdot A_{m-1})(x,y,\theta).
\end{split}
\end{align*}
Similarly, we can see that $A_m$ given by \eqref{NewRecursive} is a particular solution to the above equation. 
\end{proof}

\begin{rmk}\label{FinalRemark} In the recursive formula \eqref{NewRecursive}, each $A_m$ depends on second order derivatives of the previous term $A_{m-1}$. This is similar to the relations between the coefficients in Hadamard's parametrix (see \cite{Ze2}) for the fundamental solution of wave equation. In Hadamard's parametrix, we have 
$$U=U_0+\Gamma U_1+\cdots+\Gamma^m U_m+\cdots,$$
where each $U_m(x,t)$ is a smooth function defined on $W\times \mathbb{R}^+$ and $W$ is some neighborhood on the given Riemannian manifold $(M,g)$. The function
$U_0$ is smooth and non-vanishing. For $m\geq 1$,
\begin{align}\label{Hadamard}
U_m=-\frac{U_0}{4ms^{n+m+1}}\int_0^sU_0^{-1}s^{n+m}\Box U_{m-1}ds,
\end{align}
where $\Box=\frac{\partial^2}{\partial t^2}+\Delta_g$. Hadamard's result says that when $(M,g)$ is real analytic, $U_m$ is dominated by some $C^m$. But the same procedure does not work for $A_m$ because we do not have the $m$ appearing in the denominator of \eqref{Hadamard}, nor the factor $s^{n+m}$ in the integrand, which contributes as $\frac{1}{s^{n+m+1}}\int_0^ss^{n+m}ds=\frac{1}{n+m+1}$. With these differences in mind, it seems that $A_m$ can only be controlled by $C^mm!^2$. We obtain the same bounds for $b_m$ by the relation between $A_m$ and $b_m$ in \eqref{BandA}. 
\end{rmk}

\section*{Acknowledgements}

The authors are thankful to Steve Zelditch for stimulating conversations about the Bergman kernel coefficients.  The third author thanks Prof. Bernard Shiffman for his constant support and mentoring. The third author would also like to thank Prof. Yuan Yuan and Prof. Hao Xu for their friendly discussions and useful suggestions.


\begin{thebibliography}{HHHH}

\bibitem[BeBeSj08]{BBS}
	Berman, R., Berndtsson, B., Sj\"ostrand, J.,
	\emph{A direct approach to Bergman kernel asymptotics for positive line bundles}.
	Ark. Mat. \textbf{46(2)}, 197--217 (2008). 

\bibitem[Be03]{Bern} Berndtsson, B. \emph{Bergman kernels related to Hermitian line bundles over compact complex manifolds}, Explorations in complex and Riemannian geometry,
Contemp. Math., \textbf{332},  1--17, Amer. Math. Soc., Providence, RI, 2003. 

\bibitem[Ber08]{Ber}
Berman, R.,
\emph{Sharp asymptotics for {T}oeplitz determinants and convergence
	towards the {G}aussian free field on {R}iemann surfaces}.
Int. Math. Res. Not. \textbf{22}, 5031--5062 (2012).

\bibitem[Be10]{B}
	Berndtsson, B.
	\emph{An introduction to things $\dbar$. Analytic and Algebraic Geometry}, McNeal, 7--76 (2010).

\bibitem[BlShZe00]{BSZ} Bleher, P., Shiffman, B., Zelditch, S., \emph{Universality and scaling of correlations between zeros on complex manifolds}.
Invent. Math. \textbf{142} (2), 351--395 (2000).

\bibitem[Bo47]{Boc} Bochner, S.
\emph{Curvature in {H}ermitian metric}. Bull. Amer. Math. Soc. 179--195 (1947). 

\bibitem[BoSj75]{BoSj} Boutet de Monvel, L.,  Sj\"ostrand, J.
\emph{Sur la singularit\'e des noyaux de Bergman et de Szeg\"o}. \'Equations aux D\'eriv\'ees Partielles de Rennes, Asterisque \textbf{34-35}, 123--164 (1976), Soc. Math. France, Paris.


\bibitem[Cal53]{Cal} Calabi, E.,
\emph{Isometric imbedding of complex manifolds.}
Ann. of Math. \textbf{58} (2), 1--23 (1953).

\bibitem[Ca99]{Ca}
	Catlin, D.,
	\emph{The Bergman kernel and a theorem of Tian. Analysis and Geometry in Several Complex Variables}, Katata,
	Trends Math., 1--23. Birkh\"auser, Boston (1999).

\bibitem[Ch91]{Ch1} Christ, M. \emph{On the $\bar \partial$ equation in weighted $L^2$ norms in $\mathbb C^1$.} J. Geom. Anal., \textbf{1}(3), 193--230 (1991).

\bibitem[Ch03]{Ch} Christ, M.,
\emph{Slow off-diagonal decay for Szeg\"o kernels associated to smooth Hermitian line bundles.} Harmonic analysis at Mount Holyoke (South Hadley, MA, 2001), 77--89,
Contemp. Math., \textbf{320}, Amer. Math. Soc., Providence, RI, 2003.

\bibitem[Ch13a]{Ch2} Christ, M., \emph{Upper bounds for Bergman kernels associated to positive line bundles with smooth Hermitian metrics}, unpublished (2013), arXiv:1308.0062. 

\bibitem[Ch13b]{Ch3} Christ, M., \emph{Off-diagonal decay of Bergman kernels: On a conjecture of Zelditch}, unpublished (2013), arXiv:1308.5644. 

\bibitem[DaLiMa06]{DLM}  Dai, X., Liu, K., Ma, X., \emph{On the asymptotic expansion of {B}ergman kernel.} J. Differential Geom. \textbf{72}, 1--41(2006).

\bibitem[Do01]{Do} Donaldson S. K. \emph{Scalar curvature and projective embeddings. I}, J. Differential Geom. \textbf{59}(3), 479--522 (2001).

\bibitem[De98]{Del} Delin, H. \emph{Pointwise estimates for the weighted Bergman projection kernel in $\mathbb C^n$, using a weighted $L^2$ estimate for the $\bar \partial$ equation.} Ann. Inst. Fourier (Grenoble), \textbf{48}(4), 967--997 (1998).

\bibitem[En00]{En}  Engli\v s, M., \emph{The asymptotics of a Laplace integral on a K\"ahler manifold.} Trans. Amer. Math. Soc. \textbf{528}, 1--39 (2000).

\bibitem[Fe74]{Fe}
	Fefferman, C.,
	\emph{The Bergman kernel and biholomorphic mappings of psuedoconvex domains.}
	Invent. Math. \textbf{26}, 1--66 (1974).


\bibitem[HeKeSeXu16]{HKSX} Hezari, H., Kelleher, C., Seto, S., Xu, H., \emph{Asymptotic expansion of the Bergman kernel via perturbation of the Bargmann-Fock model,} Journal of Geometric Analysis, \textbf{26}(4), 2602--2638 (2016). 

\bibitem[Ho66]{Ho} H\"ormander, L., \emph{An introduction to complex analysis in several variables.} D. Van Nostrand Co., Inc., Princeton, N.J.-Toronto, Ont.-London 1966.



\bibitem[KaSc01]{KS} Karabegov, A., Schlichenmaier, M. \emph{Identification of Berezin-Toeplitz quantization}, J. Reine Angew. Math. \textbf{540}, 49--76 (2001).


\bibitem[Lin01]{Lindholm} Lindholm, N. \emph{Sampling in weighted $L^p$ spaces of entire functions in $\mathbb C^n$ and estimates of the
Bergman kernel.} J. Funct. Anal. \textbf{182}(2), 390--426 (2001).

\bibitem[Liu10]{Liu}
	Liu, C.-J.,
	\emph{The asymptotic {T}ian-{Y}au-{Z}elditch expansion on {R}iemann surfaces with constant curvature},
	Taiwanese J. Math., 1665--1675 (2010). 

\bibitem[LiuLu15]{LiuLu1}
	Liu, C.-J., Lu, Z.,
	\emph{Uniform asymptotic expansion on {R}iemann surfaces}. Analysis, complex geometry, and mathematical physics: in honor of {D}uong {H}. {P}hong,
	Contemp. Math. \textbf{644}, 159--173 (2015). 
	
\bibitem[LiuLu16]{LiuLu2}
	Liu, C.-J.,  Lu, Z.,
	\emph{Abstract {B}ergman kernel expansion and its applications},
	Trans. Amer. Math. Soc. \textbf{368}, 1467--1495 (2016). 
 
\bibitem[Lo04]{Loi} Loi, A.,  \emph{The Tian-Yau-Zelditch asymptotic expansion for real analytic K\"ahler metrics.} Int. J. Geom. Methods Mod. Phys. \textbf{1}(3), 253--263 (2004).

\bibitem[Lu00]{Lu}
	Lu, Z.,
 	\emph{On the Lower Order Terms of the Asymptotic Expansion of Tian-Yau-Zelditch},
	American Journal of Mathematics, \textbf{122}(2), 235--273 (2000). 

\bibitem[LuSe17]{LS} Lu, Z.,  Seto S. \emph{Agmon type estimates of the Bergman Kernel for non-compact manifolds}, preprint, 2017.

\bibitem[LuSh15]{LuSh}
	Lu, Z., Shiffman, B.,
	\emph{Asymptotic Expansion of the Off-Diagonal Bergman Kernel on Compact \k Manifolds},
	Journal of Geometric Analysis, \textbf{25}(2), 761--782 (2015). 

\bibitem[LuTi04]{LuTian}
	Lu, Z., Tian, G.,
 	\emph{The log term of the {S}zeg\"o kernel},
	Duke Math. J., \textbf{125}(2), 351--387 (2004).
	
\bibitem[LuZe16]{LuZe}
	Lu, Z., Zelditch, S.,
 	\emph{Szeg\"o kernels and {P}oincar\'e series},
	J. Anal. Math., \textbf{130}, 167--184 (2016).	

\bibitem[MaMa07]{MaMaBook} Ma, X. and  Marinescu, G., \emph{ Holomorphic Morse inequalities and Bergman kernels}, Progress in Math., \textbf{254}, Birkh\"auser, Basel, 2007.

\bibitem[MaMa08]{MaMa}
	Ma, X., Marinescu, G.,
 	\emph{Generalized {B}ergman kernels on symplectic manifolds},
	Adv. Math., \textbf{217}(4), 1756-1815 (2008).


\bibitem [MaMa13]{MaMaOff} Ma, X., Marinescu, G., \emph{Remark on the Off-Diagonal Expansion of the Bergman Kernel on Compact K\"ahler Manifolds}, Communications in Mathematics and Statistics, \textbf{1}(1), 37--41 (2013).


\bibitem[MaMa15]{MaMaAgmon} Ma, X., Marinescu, G. \emph{Exponential estimate for the asymptotics of Bergman kernels.}
Math. Ann. \textbf{362}(3-4), 1327--1347 (2015). 

\bibitem[Se15]{Seto} Seto, S. \emph{On the asymptotic expansion of the Bergman kernel}, Thesis (Ph.D.)-University of California, Irvine. (2015).


\bibitem[ShZe02]{ShZe}
 Shiffman, B., Zelditch, S.,
 \emph{Asymptotics of almost holomorphic sections of ample line bundles on symplectic manifolds}, J. Reine Angew. Math. \textbf{544}, 181-222 (2002). 

\bibitem[Sj82]{Sj} Sj\"ostrand, Singularit\'es analytiques microlocales, Ast\'erisque,  \textbf{95}(1982), 1--166, Soc. Math. France, Paris.

\bibitem[Ti90]{Ti}
	Tian, G.,
	\emph{On a set of polarized \k metrics on algebraic manifolds}, J. Differ. Geom. \textbf{32}(1), 99--130 (1990).

\bibitem[Xu12]{Xu} Xu, H.,
\emph{A closed formula for the asymptotic expansion of the Bergman kernel.}
Comm. Math. Phys. \textbf{314}(3), 555--585 (2012).

\bibitem[YuZh16]{YZ} Yuan, Y., Zhu, J.
\emph{Holomorphic line bundles over a tower of coverings.}
J. Geom. Anal. \textbf{26}(3), 2013--2039 (2016).

\bibitem[Ze98]{Ze1}
	Zelditch, S., \emph{Szeg\"{o} kernels and a theorem of Tian}, Internat. Math. Res. Notices \textbf{6}, 317--331 (1998). 

\bibitem[Ze09]{ZeBookReview} Zelditch, S. \emph{Book review of "Holomorphic Morse inequalities and Bergman kernels" (by Xiaonan Ma and George Marinescu)}. Bulletin of the American Mathematical Society \textbf{46}, 349--361 (2009).

\bibitem[Ze12]{Ze2}
	Zelditch, S.,
 	\emph{Pluri-potential theory on {G}rauert tubes of real analytic {R}iemannian manifolds, {I}}. Spectral geometry, Proc. Sympos. Pure Math. \textbf{84}, 299--339 (2012). 
	
\bibitem[Ze16]{Ze3}  Zelditch, S.,  \emph{Off-diagonal decay of toric Bergman kernels}, Lett. Math. Phys. \textbf{106}(12), 1849--1864 (2016). Volume in Memory of Louis Boutet de Monvel. 



\end{thebibliography}

\end{document}